\documentclass[reqno]{amsart}
\usepackage{graphicx} 
\usepackage{amsmath}
\usepackage{amssymb}
\usepackage{amsthm}
\usepackage{amsfonts}
\usepackage{graphicx}
\graphicspath{ {pictures/} }
\usepackage{float}
\usepackage[export]{adjustbox}
\usepackage{wrapfig}
\usepackage{mathrsfs}
\usepackage{xcolor}
\usepackage{hyperref}
\usepackage[shortlabels]{enumitem}

\usepackage[style=numeric]{biblatex}
\newtheorem{theorem}{Theorem}[section]
\newtheorem{proposition}[theorem]{Proposition}
\newtheorem{corollary}[theorem]{Corollary}
\newtheorem{lemma}[theorem]{Lemma}

\newtheorem*{theorem*}{Theorem}
\newtheorem*{BT}{Besicovitch's Theorem}
\newtheorem*{BD}{Besicovitch Density}
\newtheorem*{DMMin}{Dense Monotone Minimum}

\theoremstyle{definition}
\newtheorem{definition}[theorem]{Definition}

\newcommand{\N}{\mathbb{N}}
\newcommand{\Cant}{2^{\N}}
\newcommand{\Str}{2^{<\N}}

\newcommand{\RCA}{\mathsf{RCA}}
\newcommand{\ACA}{\mathsf{ACA}}
\newcommand{\WKL}{\mathsf{WKL}}
\newcommand{\WWKL}{\mathsf{WWKL}}
\newcommand{\SOA}{\mathsf{Z}}
\newcommand{\BCTC}{\mathsf{BCTC}}
\newcommand{\DMM}{\mathsf{DMMin}}
\newcommand{\BES}{\mathsf{BES}}
\newcommand{\CA}{\mathsf{CA}}

\title[Baire Category Approach to Besicovitch's Theorem]{A Baire Category Approach to Besicovitch's Theorem and Measure Regularity}
\author[E. Gruner]{Emma Gruner}
\address{Department of Mathematics \\ Penn State University \\ University Park, Pennsylvania \\ 16802}
\email{eeg67@psu.edu}
\author[J. Reimann]{Jan Reimann}
\address{Department of Mathematics \\ Penn State University \\ University Park, Pennsylvania \\ 16802}
\email{jsr25@psu.edu}

\begin{document}



\begin{abstract}
    By reformulating the classical proof as a Baire Category argument, we show that Besicovitch's Theorem on the existence of subsets of finite Hausdorff measure is provable in $\ACA_0$, and additionally that the witnessing subset is computable from one jump of the original set. We show that the corresponding formulation of Baire Category, which we call Baire Category Theorem for Closed Sets ($\BCTC$), is equivalent to $\ACA_0$, contrasting with previous results on the reverse math strength of Baire Category variants. We also examine the implications of $\BCTC$ for a class of monotone functions on closed sets, and explore how changing the representation of a closed set affects the reverse math strength of its measure regularity properties.
\end{abstract}

\maketitle

\section{Introduction}

Since its inception in the latter half of the twentieth century, the reverse mathematics research program has been used to analyze theorems across a wide range of disciplines. Real analysis and point-set topology have been some of the most popular areas of study, since many of the objects involved (real numbers, separable metric spaces, continuous functions, open sets, etc.) have ``essentially countable" structures which allow them to be coded naturally as elements of $\Cant$. Thus, these fields can be meaningfully discussed in the context of second order arithmetic.

Geometric measure theory, on the other hand, is not widely explored at all in the reverse (or computational) mathematics framework. Reimann~\cite{reimann2008effectively} studied some effective aspects of Frostman's Lemma. Later, Pauly and Fouché~\cite{pauly2017constructive} classified Frostman's Lemma in the context of the Weihrauch degrees. 

Besides Frostman's Lemma, another foundational result of geometric measure theory is a theorem proven by Besicovitch in 1952 \cite{besicovitch1952existence},  which states that any closed subset in Euclidean space having infinite Hausdorff measure contains a closed subset with positive finite Hausdorff measure. Sets of positive finite measure have many convenient local properties (see \cite{falconer2013fractal} for an extensive treatment), so the ability to ``pass" to such a subset is highly useful. Later in 1952 \cite{davies1952accessibility}, Davies strengthened Besicovitch's result to the case where the given set is analytic. 

The main motivation for this paper is to study the set existence axioms required to prove Besicovitch's theorem.  
Besicovitch's original argument indicates a rather high complexity. If $F$ represents the original closed subset, then the witnessing subset $E$ is defined as the infinite intersection of nested, decreasing closed subsets $\{E_k\}_{k \in \omega}$. By the properties of the $s$-dimensional Hausdorff $\delta$-measure on compact sets, one can verify that $E_0$ has an arithmetic definition from $F$, and each subsequent $E_{k+1}$ can be defined similarly from $E_k$. Thus, while the existence of each set $E_k$ in the approximating sequence is provable in $\ACA_0$, the same may not be true for the infinite intersection $E = \bigcap_{k \in \omega}E_k$.

We impose a topology on the space of closed subsets of our given set $F$, and relate the Hausdorff $\delta$-measures of these sets to certain open and closed sets within this topology. By exploiting this topological structure, we use a variant of the Baire Category Theorem to significantly decrease the arithmetic complexity of the final witnessing subset $E$, yielding the following theorem:

\begin{theorem*}
    In Cantor space $2^{\mathbb{N}}$, Besicovitch's Theorem is provable in $\ACA_0$.
\end{theorem*}

Extracting the computational content, our proof of the theorem above yields the following:

\begin{theorem*}
    Let $F \subseteq 2^{\omega}$ be a closed subset with infinite $s$-dimensional Hausdorff measure, and suppose $Z_F \in 2^{\omega}$ is a code for its tree. Then we can find a closed subset $E \subseteq F$ with positive finite Hausdorff measure whose code $Z_E$ is computable from $Z_F'$ (the first Turing jump of $Z_F$).
\end{theorem*}

In terms of the \emph{instance-solution} framework, this means Besicovitch's Theorem is computably reducible to the Turing jump ($\mathsf{TJ}$).

\medskip
As outlined above, Besicovitch's original argument suggests that infinitely many jumps of the original closed set $F$ may be required to compute the witnessing subset $E$. The theorems above also stands in contrast to a result by Kjos-Hanssen and Reimann regarding $G_\delta$ sets~\cite{kjos2014finding}. By the full Besicovitch-Davies Theorem, any $G_{\delta}$ set of infinite Hausdorff measure contains a closed subset of positive measure, but the authors exhibited a lightface $\Pi^0_2$ subset of $2^{\omega}$ whose witnessing closed subsets could not be computed from any hyperarithmetic real.

Whether Besicovitch's Theorem is equivalent to $\ACA_0$ remains open. However, we do show that the theorem implies the axioms of $\WWKL_0$, or even $\WKL_0$ if we impose additional structure on the witnessing subset $E$. (See Section~\ref{ssec:besicovitch} for details). 

\medskip
While Besicovitch's Theorem provided the main motivation for this project, the measure theoretic and topological techniques developed along the way lead to interesting avenues of study in their own right. For one, the definition of $s$-dimensional Hausdorff measure involves taking the limit of a sequence of outer measures, called $\delta$-measures. If we fix one of these outer measures, we can ask a variety of questions related to its regularity properties and their reverse mathematical strength. These investigations also have implications for the standard Lebesgue measure on $2^{\omega}$, since any of these $s$-dimensional $\delta$-measures reduces to the Lebesgue measure when $s = 1$.

Several works on the reverse mathematics of measure theory have already explored many of its technicalities. Yu~\cite{yu1987measure} and subsequently Yu and Simpson~\cite{yu1990measure} showed that $\ACA_0$ was necessary to prove that the Lebesgue measure of every open set exists and that $\WWKL_0$ was necessary to prove that the Lebesgue measure is countably additive. Simpson~\cite{simpson2009mass} investigated the reverse math strength of the regularity property that any Borel set contains an $F_{\sigma}$-subset of the same Lebesgue measure \cite{simpson2009mass}. Surprisingly, this property turns out to be independent of many standard subsystems of second order arithmetic, as one can find models of $\RCA_0$ where this regularity property holds, but where the axioms of $\WWKL_0$, $\WKL_0$, and $\ACA_0$ may either hold or fail. 

In this paper, we restrict our attention to closed sets, and to questions that are some variation of the following: given a closed subset $F$, how hard is it to construct a closed subset $E$ of some desired $s$-dimensional $\delta$-measure? In particular, we explore how the answer to this question may change depending on how the closed sets are coded; we introduce the notions of ``standard closed" and ``pruned closed" to describe differences in the structure of these sets' representative trees. Other variables that affect the reverse math strength are the precision with which we specify the target measure of our set $E$, and whether the measure of our starting set $F$ exists in the space. In Section~\ref{ssec:regularity}, we determine the strength of various regularity properties for $\delta$-measures on closed sets with respect to these parameters. (Preliminaries regarding the representations of closed sets and the Hausdorff $\delta$-measure are given in Sections~\ref{ssec:closed} and~\ref{ssec:Hmeas} respectively.) 

\medskip
Our main tool to bound the complexity of Besicovitch's Theorem is Baire Category. In its most straightforward formalization, Baire Category is a rather ``simple" result, being provable in $\RCA_0$. (See \cite{simpson2009subsystems} for a proof). 
However, it was later discovered that the exact manner in which the open sets in the statement are coded is important. Besides the standard coding of an open set as a countable union of basic sets, one can also represent an open set by specifying a countable dense sequence of points in its closed complement. Such a representation is called a ``separably open" set, and was explored by Brown~\cite{brown1990notions}. Brown showed that $\Pi^1_1\text{-}\CA_0$ is required to prove the equivalence of these two representations for arbitrary metric spaces, although $\ACA_0$ suffices for compact spaces. Later, Simpson and Brown~\cite{brown1993baire} showed that replacing the standard open set representations with separably open sets yields a version of Baire Category which is not provable in $\RCA_0$ or even $\WKL_0$. However, this new Baire Category variant did not require the full power of $\ACA_0$, since it could be proven in a strictly weaker subsystem that the authors called $\RCA_0^{+}$. (Mytilinaios and Slaman~\cite{Mytilinaios:1996a} later showed that the version of the Baire Category Theorem for separably open sets is actually weaker than $\RCA_0^{+}$.)

In addition to reverse mathematics, different variations of the Baire Category theorem have also been studied in the context of Weihrauch reducibility. Baire Category lends itself naturally to this setting. In a 2018 paper~\cite{Brattka:2018a}, Brattka, Hendtlass, and Kreuzer explored four main variations of the Baire Category theorem, as well as their connections with some other results on genericity. Two of these variations, which the authors denote $\text{BCT}_0$ and $\text{BCT}_2$, are roughly equivalent to the formalizations of Baire Category using standard open sets and separably open sets respectively. The authors show that $\text{BCT}_2$ is strongly Weihrach equivalent to $\text{BCT}_0'$. This is consistent with the result from the Brown-Simpson paper, in the sense that the coding with separably open sets generates a computationally ``stronger" version of Baire Category than its usual presentation. Brattka et al.\ also observe a similar phenomenon with their formalizations of yet another statement of Baire Category: ``For every countable sequence of closed sets whose union is the entire metric space, there exists at least one set in the sequence with nonempty interior." 

\medskip
In this paper, we will interpret the Baire Category Theorem in the context of a closed subset of Cantor space. Since a closed subset of a complete metric space is itself complete under the same metric, the Baire Category Theorem applies when working with open sets in the subspace topology. We call this variant ``Baire Category Theorem for Closed Sets" ($\BCTC$), and can formally state it as follows:

\begin{quote}
    \textit{Let $X$ be a complete metric space, and let $F \subseteq X$ be a nonempty closed set. Suppose $\{U_n\}_{n \in \N}$ is a sequence of nonempty open sets in $X$ such that each $U_n$ is dense in $F$ (that is, if $N \subseteq X$ is a basic open set such that $N \cap F \neq \emptyset$, then there exists some $x \in N \cap F \cap U_n$). Then $\bigcap_{n \in \N}U_n$ is also dense in $F$.}
\end{quote}

Like the classical Baire Category Theorem, we must be explicit about how the open and closed sets in the statement are coded in order to analyze the reverse mathematical strength. The most natural interpretation for Cantor space, and that which is necessary for the proof of Besicovitch's Theorem, is to code the open sets $U_n$ as unions of cylinder sets and the closed set $F$ as a tree. In this manner, the classical proof of Baire Category in $\RCA_0$ does not translate, since for a given basic open set $N$, there is no computable way to determine whether $N \cap F$ is nonempty. Given the previous result of Brown and Simpson regarding the ``separably open'' version of Baire Category, one might conjecture that $\BCTC$ is still relatively simple from a reverse math perspective. However, we are able to prove the following equivalence:

\begin{theorem*}
    Over $\RCA_0$, $\BCTC$ is equivalent to $\ACA_0$.
\end{theorem*}

This equivalence only holds in the case where the closed set $F$ is ``standard closed." Taking $F$ to be a separably closed or pruned closed set gives a version of $\BCTC$ which is still provable in $\RCA_0.$ See Section~\ref{ssec:bct} for details.

In Section~\ref{ssec:monotone}, we explore how the closed version of Baire Category Theorem has some interesting implications regarding certain monotone functions defined on closed subsets of Cantor space, particularly those which satisfy a certain density property regarding their behavior around their infimum. We define the \emph{Dense Monotone Minimum Principle} ($\DMM$) to be the statement that a function with this property always realizes its infimum. We analyze the reverse math strength of this principle.

\begin{theorem*} \label{thm:intro-DMMin}
    For $\omega$-models of $\RCA_0$, $\DMM$ is equivalent to $\ACA_0$.
\end{theorem*}

One of the most natural examples of a monotone function with this density property is precisely the $s$-dimensional Hausdorff $\delta$-measure, when restricted to the set of closed sets whose measures satisfy some fixed lower bound. 
As such, $\DMM$ also serves an important role in the intermediate steps of the Besicovitch Theorem proof.

\medskip
\subsection*{Formalization and working with \texorpdfstring{$\omega$}{omega}-models}
When working over $\RCA_0$, some of our proofs, reversals in particular,  will only work for $\omega$-models. This happens for various reasons -- in some cases $\RCA_0$ is too weak to formalize some concepts (like Hausdorff measures), in other cases we are using techniques that are not provable in $\RCA_0$, such as the infinitary pigeonhole principle or induction over higher order formulas. We could work over a stronger base theory to avoid formalization issues, but in most cases this base theory would have to be $\ACA_0$, and we would lose the fine-grain picture we are able to obtain when working with $\omega$-models. Whenever this is the case, we will state our results using \emph{$\omega$-reducibility}, $\leq_\omega$. When adapted to the \emph{instance-solution} framework, many results will actually yield the stronger \emph{computable reducibility}, $\leq_c$ (see~\cite{hirschfeldt2015slicing} or\cite{Dzhafarov:2022a} for background).   

In any case, $\ACA_0$ is strong enough to formalize and prove all the statements we are developing, so that in Section~\ref{ssec:besicovitch} we can use them to show that Besicovitch's Theorem is provable in $\ACA_0$.

\section{Background and Preliminaries}

\subsection{Subsystems of Second Order Arithmetic}

We assume some standard background in reverse mathematics and second order arithmetic; in particular, we will assume that the reader is familiar with the definitions and key properties of the subsystems $\RCA_0,$ $\WKL_0$, $\WWKL_0$ and $\ACA_0$. More information can be found in any standard text on reverse mathematics (see, for example, \cite{simpson2009subsystems}, \cite{hirschfeldt2015slicing}, and~\cite{Dzhafarov:2022a}).  

As is customary, $\omega$ will denote the set of natural numbers $\{0,1,2,\dots\}$, $2^{<\omega}$ will denote the set of finite binary strings, and $2^\omega$ will denote the set of infinite binary stings, but when working with formal statements in subsystems of second order arithmetic $\SOA_2$, we will use symbols $\N$, $\Str$, and $\Cant$, respectively, to denote the formalized versions instead.

\subsection{Closed Sets and their Representations}
\label{ssec:closed}

In this paper, we will concern ourselves with measure and topological results in Cantor space $2^{\omega}$. We will use the notation $N_{\sigma}$ to denote the basic open set 
\[
N_{\sigma} = \{X \in 2^{\omega} : X \supseteq \sigma \},
\]
where $\sigma \in 2^{<\omega}$. It is likely that many of the (non-measure theoretic) results in this work can be generalized to arbitrary compact metric spaces, although we have not formalized any details here. Regarding Hausdorff measure, Besicovitch's Theorem becomes significantly harder  to prove for arbitrary compact metric spaces, due to the absence of a canonical ``net structure'' that both Cantor space and Euclidean possess, and determining a good upper bound on the complexity of the theorem remains an open problem. 

\medskip
To code closed subsets $F \subseteq \Cant$ as a subset of $\N$, we will almost exclusively use trees, i.e., subsets $T \subseteq \Str$ closed under initial segments. A \textit{pruned tree} has the additional property that for every $\sigma \in T$, at least one of $\sigma^{\frown}0 \in T$ or $\sigma^{\frown}1 \in T$. The set of all (infinite) paths through a tree $T$ is denoted by $[T]$.

Classically, any closed set $F \subseteq 2^{\omega}$ can be represented by some tree $T \subseteq 2^{<\omega}$. However, in second order arithmetic we equate a closed set with a tree coding it, which may or may not exist. The following definitions should therefore be interpreted within $\RCA_0$.

\begin{definition}
    A set $F \subseteq \Cant$ is 
    \begin{enumerate}[(a)]
        \item \textit{standard closed} if there exists a tree $T \subseteq \Str$ with $F = [T]$,
        \item \textit{pruned closed} if there exists a pruned tree $T \subseteq \Str$ with $F = [T]$.
    \end{enumerate}
\end{definition}

Certainly any closed set which has a pruned tree representation also has a standard tree representation, since a pruned tree is just a specific type of tree. However, depending on the axiom system being used, the converse does not always hold; we will explore the reverse math strength of the existence of pruned tree representations later in this section. Note that unlike the standard tree representation, the pruned tree representation for a closed set $F \subseteq \Cant$ is unique, consisting of exactly the strings $\sigma \in \Str$ which are initial segments of some element $X \in F$.

Also, later in this work it will be convenient to have a name for closed subsets whose corresponding tree representations are infinite. If we are working in an axiom system equivalent to or stronger than $\WKL_0$, this is equivalent to the closed set being nonempty. However, to avoid any ambiguity, we make the following definition.

\begin{definition}
    Let $F$ be a standard closed or pruned closed set in $\Cant.$ We call $F$ \textit{nontrivial} if its corresponding tree (or pruned tree) in $\Str$ is infinite.
\end{definition}

An alternative way to code a closed set is by specifying a countable sequence of elements of $\Cant$ which are dense in that set, in the following sense:

\begin{definition}
Let $S = \langle X_n \rangle_{n \in \N}$ be a sequence of elements of $\Cant$. The set $\bar{S}$ consists of all $X \in \Cant$ with the property that
\[
\forall k \: \exists n  \: \forall i < k \; X_n(i) = X(i).
\]
A set $F \subseteq \Cant$ is called \textit{separably closed} if it is empty, or if there exists a sequence $S$ of elements of $\Cant$ such that $F = \bar{S}$.
\end{definition}

When it comes to representing open sets in $\Cant$, we will use the notion defined below. This corresponds naturally to the idea of an open set as a union of cylinder sets.

\begin{definition}
    A set $U \subseteq \Cant$ is called $\textit{open}$ if there exists a set $V \subseteq \Str$ with the property that for all $X \in \Cant$,
    \[
    X \in U \iff \exists \sigma \in V \: \sigma \subseteq X.
    \]
    Such a set $V$ is called a \textit{code} for $U$. Without loss of generality, we may assume that $V$ is closed upward under extensions: that is, if $\sigma \in V$ and $\tau \supseteq \sigma$, then $\tau \in V$. 
\end{definition}

One can also represent an open set as the complement of a separably closed set. While we will not work with this representation here, it has been previously studied in the context of Baire Category, as mentioned in the introduction.

\medskip
It turns out that the ability to translate between different representations of the same closed set (or open set) is not trivial, as the results below will demonstrate. In the statements of these results, we will abuse language slightly for the sake of efficiency, as we will conflate the existence of closed subsets with the existence of their appropriate representations. For example, the statement ``Every separably closed set is pruned closed" is technically shorthand for ``For every countable sequence $S$ of elements of $\Cant$, there exists a pruned tree $T \subseteq \Str$ such that for all $X \in \Cant$, $X \in \bar{S}$ if and only if $X \in [T]$."

First, we note that a pruned tree representation is the strongest in a sense, since it can generate both of the remaining representations.

\begin{proposition} \label{RCA-P}The following are provable in $\RCA_0:$ 
\begin{enumerate} [(i)]
    \item Every pruned closed set in $\Cant$ is standard closed.
    \item Every pruned closed set in $\Cant$ is separably closed.
\end{enumerate}
\end{proposition}

\begin{proof}[Proof $($sketch$)$]
    (i) is clear.
    For (ii), let $T$ be a pruned tree. If $T$ is finite, it must be empty and the corresponding closed set $F = [T]$ is empty, too, so trivially separably closed.
    So let $T$ be infinite. It is provable in $\RCA_0$ (see e.g.~\cite{Dzhafarov:2022a}) that there is an enumeration $\{\sigma_n\}_{n \in \N}$ of the elements of $T$. (Strictly speaking, $T$ is the range of a one-one function from $\N$ onto $T$.) Using primitive recursion, we can define a sequence $S = \langle X_n \rangle_{n \in \N}$ in $\Cant$ where $X_n$ is the leftmost extension of the string $\sigma_n$ through $T$. It is then straightforward to verify that $\bar{S}$ and $[T]$ define the same closed set.
\end{proof}

In contrast to the result above, obtaining a pruned tree representation from a standard or separately closed set, or translating between the latter two representations, requires more axiomatic strength. Variations of items (iii) and (iv) of the theorem below have been proven before for arbitrary compact metric spaces (see \cite{brown1990notions}). 

\begin{theorem} \label{ACA-T}
    Over $\RCA_0$, the following statements are all equivalent to $\ACA_0$:
    \begin{enumerate}[(i)]
        \item Every standard closed set in $\Cant$ is pruned closed. 
        \item Every separably closed set in $\Cant$ is pruned closed. 
        \item Every standard closed set in $\Cant$ is separably closed. 
        \item Every separably closed set in $\Cant$ is standard closed. 
    \end{enumerate}
\end{theorem}

\begin{proof}
    We show that $\ACA_0$ is sufficient to prove both (i) and (ii). Proposition~\ref{RCA-P} then implies that it can also prove (iii) and (iv). The proof of the reversal for (iii) and (iv) in~\cite{brown1990notions} can be adapted to our setting to complete the proof of the theorem (see also~\cite{Gruner:2026a}).
    
     $\ACA_0 \implies (i)$:
    Let $T \subseteq \Str$ be a tree. Using arithmetic comprehension, we can define the tree $\tilde{T}$ by
    \[
    \tilde{T} = \{\sigma \in T: \forall n \geq |\sigma| \: \exists \tau \supseteq \sigma \text{ such that } |\tau| = n \text{ and } \tau \in T\}.
    \]
    $\tilde{T}$ is closed under initial segments. It is also pruned: suppose to the contrary that there was some $\sigma \in \tilde{T}$ for which neither of $\sigma^\frown 0$ or $\sigma^\frown 1$ are in $\tilde{T}$. Unpacking the definition, this would imply that there was some $n_0 \geq |\sigma^\frown 0|$ for which no length-$n_0$ extensions of $\sigma^\frown 0$ were in $T$; we can choose $n_1$ for $\sigma^\frown 1$ similarly. Consider $n= \max\{n_0,n_1\}$. Since $\sigma \in \tilde{T}$, we know there is some length-$n$ extension $\tau \supseteq \sigma$. Since $n$ must be strictly greater
than $|\sigma|$, this string $\tau$ must extend one of $\sigma^\frown 0$ or $\sigma^\frown 1$.
In either case, $\tau\restriction_{n_i}$ would be a length-$n_i$ extension of $\sigma^\frown i$, and this element would be
in $T$ by initial segment closure, contradiction.
    
    It is  straightforward to verify that $[T]$ and $[\tilde{T}]$ represent the same closed set.

\medskip
    $\ACA_0 \implies (ii)$:
     If we have a separably closed set which is empty, then the empty tree $T$ is trivially a pruned tree representing the same closed set. So assume our separably closed set $F$ is nonempty, and let $S = \langle X_n \rangle_{n \in \N}$ be a sequence representing $F$. Using arithmetic comprehension, we can define a pruned tree $T$ by collecting all the finite initial segments of the elements in the sequence $S$:
     \[
     T = \{ \sigma \in \Str: \exists n \text{ such that } \sigma \subseteq X_n \}.
     \]
    It is clear that $T$ is closed under initial segments. Moreover, if $\sigma \in T$ via some $X_n$, then $\sigma^\frown X_n(|\sigma|)$ is also in $T$. Thus, $T$ is a pruned tree. 

    Again, checking that $T$ represents the same closed set as $\bar{S}$ is straightforward.
\end{proof}

\subsection{Spaces of closed subsets of a closed set}

While we will briefly look at separably closed sets in the context of the Baire Category theorem (see Section~\ref{ssec:bct}), most of this paper will be concerned with standard and pruned closed sets (i.e. those which are represented by trees). These representations provide a natural way of approximating the sets' measure, as we will demonstrate in the following subsection. If we are given a particular tree representation for a standard/pruned closed set $F$, we would like a way to code the set of all nontrivial closed subsets of $F$. This can be done using trees, too. 

We fix a standard bijection between $\N$ and $\Str$ which enumerates the elements of $\Str$ first by length and then lexicographically. (It is straightforward to verify that this exists in $\RCA_0$.)

\begin{definition}
    Given an element $X \in \Cant$ and a natural number $n \in \N$, we let $X^{\leq n}$ denote the initial segment of $X$ with length $2^{n+1}-1$. If a finite string $\nu \in \Str$ is such that $|\nu| \geq 2^{n+1} - 1$, then we use the notation $\nu^{\leq n}$ similarly. 
\end{definition}

This notation is convenient since $X^{\leq n}(\sigma)$ will be defined for a particular $\sigma$ (under the bijection mentioned above) if and only if $|\sigma| \leq n$. 

\begin{definition} \label{ST}
    Let $T \subseteq \Str$ be a tree. Let $S_T \subseteq \Str$ be the tree consisting of all $\nu \in \Str$ which satisfy the following properties:
    \begin{enumerate}[(1)]
        \item For all $\sigma < |\nu|$ with $\nu(\sigma) = 1$, we also have $\nu(\tau) = 1$ for all $\tau \subseteq \sigma$.
        \item For all $\sigma <|\nu|$ with $\nu(\sigma) = 1$, we have $\sigma \in T$.
    \end{enumerate}
\end{definition}

If we have an element $Z \in [S_T]$, then we can identify $Z$ with the subset $T_Z \subseteq \Str$ consisting of exactly the strings $\sigma$ for which $Z(\sigma) = 1$; we will say that $Z$ \textit{represents} or \textit{codes} the tree $T_Z$. Property (1) in the definition of $S_T$ guarantees that $Z$ does indeed code a tree, while property (2) guarantees that this tree will be a subset of $T$. 

In $2^\omega$, if $Z \in [S_T]$ then the closed set $[T_Z]$ will be a subset of the closed set $[T]$. Conversely, if we have a tree $S \subseteq 2^{<\omega}$ such that $[S] \subseteq [T]$, then the element $Z$ representing the tree $S \cap T \subseteq 2^{<\omega}$ will be an element of $[S_T]$ which represents the same closed set as $[S]$. Thus, the paths through $S_T$ do form a reasonable correspondence with the closed subsets of $[T]$.

At times, it will be convenient to work with a variant of $S_T$ whose infinite paths all correspond to pruned trees. Such a variant is defined below.

\begin{definition} 
    Let $T \subseteq \Str$ be a tree. Let $\tilde{S}_T \subseteq \Str$ be the tree consisting of all $\nu \in \Str$ which satisfy the conditions for $S_T$ in the previous definition, as well as the following additional condition:
    \begin{enumerate}[(3)]
        \item For all $\sigma < |\nu|$ such that $\nu(\sigma) = 1$ and such that $\nu(\sigma^{\frown}0)$ and $\nu(\sigma^\frown1)$ are both defined, we have at least one of $\nu(\sigma^{\frown}0) = 1$ or $\nu(\sigma^{\frown}1) = 1$.        
    \end{enumerate}
    
\end{definition}

In light of these definitions, we can impose a completely metrizable topology on the space of nontrivial closed subsets of a particular standard/pruned closed set $F \subseteq 2^{\omega}$. If $T$ is a tree with $F = [T]$, then we can simply restrict the usual topology and metric on $2^{\omega}$ to the closed subset $[S_T]$ or $[\tilde{S}_T]$. In fact, the topology generated in this way is closely related to the Hausdorff metric on closed subsets of $2^{\omega}$ (see~\cite{Gruner:2026a}).

\subsection{Hausdorff Measure and its Representations}
\label{ssec:Hmeas}

The notion of \textit{Hausdorff measure} on a metric space is a generalization of Lebesgue measure. This measure is the main object of study in the subfield of analysis known as geometric measure theory, of which a standard reference is \cite{falconer2013fractal}. 

The classical concept of Hausdorff measure is defined as a limit of a sequence of intermediate outer measures, which can be defined in the following way for Cantor space.

\begin{definition}
    Let $E \subseteq 2^{\omega}$ be a set, let $\delta$ be a positive real number, and $s$ a nonnegative real number. We call a countable or finite collection of cylinder sets $\{N_{\sigma_i}\}_{i < N}$ a \textit{$\delta$-cover} for $E$ if $E \subseteq \bigcup_{i < N}N_{\sigma_i}$ and each $\sigma_i$ satisfies $2^{-|\sigma_i|} \leq \delta$.  (Here, $N$ can either be a finite positive integer or $\omega$.) The \textit{s-dimensional Hausdorff $\delta$-measure} $\mathcal{H}^s_{\delta}(E)$ is the quantity
    \[
    \mathcal{H}^s_{\delta}(E) = \inf\left \{ \sum_{i < N}2^{-s|\sigma_i|}: \{N_{\sigma_i}\}_{i < N} \text{ is a $\delta$-cover for } E \right \}.
    \]
\end{definition}

If $s = 1$, then $\mathcal{H}^s_{\delta}(E)$ is simply the standard Lebesgue measure $\mu(E)$ on $2^{\omega}$ regardless of the value of $\delta$. (Any cover $\{N_{\sigma_i}\}_{i < N}$ of $E$ by cylinder sets can be converted to a $\delta$-cover of $E$ without changing the value of $\sum_{i < N}2^{-|\sigma_i|}$.)

On the other hand, if $s < 1$, then different values of $\delta$ can potentially yield different values of $\mathcal{H}^s_{\delta}(E)$. However, note that if $\delta_0 > \delta_1$, then the collection of $\delta_1$-covers for a set $E$ will be a subset of the collection of $\delta_0$-covers, which implies that $\mathcal{H}^s_{\delta_0}(E) \leq \mathcal{H}^s_{\delta_1}(E)$. Therefore, the limit in the definition below is (classically) well-defined, although it can potentially be infinite. 

\begin{definition} Let  $E \subseteq 2^{\omega}$, and let $s$ be a nonnegative number. Then the \textit{s-dimensional Hausdorff measure} $\mathcal{H}^s(E)$ is defined by 
\[
\mathcal{H}^s(E) = \lim_{\delta \to 0}\mathcal{H}^s_{\delta}(E).
\]    
\end{definition}

The fact that the definition of Hausdorff measure involves checking infinitely many collections of open sets makes it somewhat nonconstructive and difficult to work with in a reverse math context. However, we can develop a more explicit measure function in the case where $E$ is closed, which will also correspond nicely to a notion of monotone functions in presented in Section~\ref{ssec:regularity}. 

As we show below, this function will be equivalent to the usual definition of Hausdorff measure as long as we are in a sufficiently strong axiom system.

\begin{definition}
    Let $V \subseteq 2^{<\omega}$ be finite, and let $s \geq 0$. Define the \textit{$s$-weight} of $V$ by 
    $$W_s(V) = \sum_{\sigma \in V}2^{-s|\sigma|}$$
    with $W_s(\emptyset)$ defined to be 0.
\end{definition}

\begin{definition} \label{def:H-tilde}
    Let $S \subseteq 2^{<\omega}$ denote the set $S_T$ from Definition \ref{ST}, where $T$ is taken to be all of $2^{<\omega}$. Suppose $s \geq 0$ and $n \in \omega$. We define a function $\tilde{H}^s_{n}: S \to \mathbb{R}$ as follows. 

Given $\nu \in S$ of length at least 1, we first find the largest integer $m$ for which $|\nu| \geq 2^{m+1} - 1$, i.e. for which $\nu^{\leq m}$ is defined. (We handle the case where $\nu$ is the empty string $\langle \rangle$ separately). We then proceed in cases depending on the value of $m$. 

\begin{itemize}
\item \textit{Case 1:} If $m \geq n$, we define 
\[
\tilde{H}^s_n(\nu) = \min \{ W_s(V)\},
\]
where the minimum is taken over all (finitely many) $V \subseteq 2^{<\omega}$ with the following properties:
\begin{enumerate}[(a)]
    \item For all $\sigma \in V$, $n \leq |\sigma| \leq m$. \label{item_a}
    \item For all $\sigma$ such that $|\sigma| = m$ and $\nu(\sigma) = 1$, there is some $\tau \in V$ with $\tau \subseteq \sigma$. \label{item_b}
\end{enumerate}

\item \textit{Case 2:} If $m < n$, or if $\nu = \langle \rangle$,  we define
\[
\tilde{H}^s_n(\nu) = 2^n\cdot 2^{-sn} = 2^{(1-s)n}.
\]
\end{itemize}
\end{definition}

It is a straightforward exercise to show that $\tilde{H}^s_n$ is computable using $s$ and $n$ as parameters. (The minimum of the, finitely many, $W_s(V)$ is computable using the definition of order and equality on real numbers when represented via Cauchy sequences.)
Hence the formalized version will make sense in any $\omega$-model of $\RCA_0$.

Observe further that if $\nu_0, \nu_1 \in S$ and $\nu_0 \subseteq \nu_1$, then $\tilde{H}^s_n(\nu_0) \geq \tilde{H}^s_n(\nu_1)$ (which is relevant to the existence of the above limit). For one, the value of $m$ for $\nu_0$ will be at most that of the value for $\nu_1$. If both strings fall under Case 1 of the definition, then whichever set $V_0$ realizes the minimum $W_s(V)$ for $\nu_0$ will still be considered among the sets $V$ for $\nu_1$, due to the ``closure under initial segments" property for the strings in $S$. Therefore, the minimum in the definition for $\tilde{H}^s_n(\nu_1)$ cannot increase from its value for $\nu_0$. 

Similarly, the set $V$ which consists of all length-$n$ strings in $2^{<\omega}$ will always be under consideration in the Case 1 definition, and $W_s(V)$ for this $V$ is precisely $2^{(1-s)n}$. Thus, the desired monotonicity still holds in the case where $\nu_0$ falls under Case 2 while $\nu_1$ falls under Case 1. If both strings fall under Case 2, then the claim is obvious.

Given this monotonicity, we are justified in making the following definition, using the function $\tilde{H}^s_n$ and the set $S$ from Definition~\ref{def:H-tilde}.

\begin{definition}
The function  $H^s_n: [S] \to \mathbb{R}_{\geq 0}$ is defined by
\[
H^s_n(Z) = \lim_{k \to \infty}\tilde{H}^s_{n}(Z \upharpoonright_k).
\]
\end{definition}

Under classical, non-restricted mathematics, the function $H^s_n$ is always well-defined for elements of $[S]$ and is precisely the $s$-dimensional Hausdorff $2^{-n}$-measure. However, given the properties of compactness and convergence necessary for the proof to proceed, we must be a bit careful when working with the formalized version in weak subsystems of second order arithmetic. 

\medskip
We would like to show that, under sufficient axiomatic strength, the definition of $H^s_n$ is consistent with that of $\mathcal{H}^s_{2^{-n}}$ for closed sets. 
We can then formulate and understand statements about the overall Hausdorff measure $\mathcal{H}^s(F)$ by interpreting it as the value of 
\[
\lim_{n \to \infty}H^s_n(Z) = \lim_{n \to \infty}\mathcal{H}^s_{2^{-n}}(F).
\]

We begin with the easier direction of this correspondence, which does not require any advanced axiomatic strength. 

\begin{proposition} \label{prop:RCA-meas} The following is satisfied in all $\omega$-models of $\RCA_0$: Let $E \subseteq \Cant$ be a standard closed set, and let $Z \in [S]$ be a code for a tree corresponding to $E$. Let $s \geq 0$ be a real number, and let $n \in \N$. Then for any initial segment $\nu \subseteq Z$, there is some $2^{-n}$-cover $\{N_{\sigma}\}_{\sigma \in V}$ of $E$ for which 
\[
\sum_{\sigma \in V}2^{-s|\sigma|} \leq \tilde{H}^s_n(\nu).
\]
\end{proposition}

\begin{proof}
Given the monotonicity of the function $\tilde{H}^s_n$, it suffices to only consider $\nu \subseteq Z$ which have the form $Z^{\leq m}$ for some $m \geq n$. For such a string $\nu$, let $V \subseteq 2^{<\omega}$ be the set realizing the minimum $s$-weight, according to the definition of $\tilde{H}^s_n(\nu)$ (Definition \ref{def:H-tilde}), and consider the cover $\{N_{\sigma}\}_{\sigma \in V}.$ Since all strings in $V$ have length at least $n$ by condition (a), this generates a $2^{-n}$-cover. 

Additionally, for any $X \in E$, all initial segments $\sigma \subseteq X$ must satisfy $Z(\sigma) = 1$, and thus $\nu(\sigma) = 1$ as long as $\nu$ is defined there. If we consider the initial segment $\sigma = X \upharpoonright_m$, then condition (b) of the definition of $\tilde{H}^s_n$ guarantees that $\sigma$ has some initial segment $\tau \in V$, which implies that $X$ is an element of the corresponding cylinder $N_\tau$ in our cover. 

Thus, $\{N_{\sigma}\}_{\sigma \in V}$ is indeed a $2^{-n}$-cover for $E$, and we have

$$\sum_{\sigma \in V}2^{-s|\sigma|} = W_s(V) = \tilde{H}^s_n(\nu).$$
\end{proof}

If we reverse the roles of $H^s_n$ and $\mathcal{H}^s_{2^{-n}}$ in the above argument, we obtain a stronger statement, reverse mathematically. 

\begin{proposition} \label{prop:WKL-meas1} The following statement is $\omega$-equivalent to $\WKL_0$: Let $E \subseteq \Cant$ be a standard closed set, and let $Z \in [S]$ be a code for a tree corresponding to $E$. Let $s \geq 0$ be a real number, and let $n \in \N$. Then for any $2^{-n}$-cover $\{N_{\sigma}\}_{\sigma \in V}$ of $E$, there exists some initial segment $\nu \subseteq Z$ such that 

$$\tilde{H}^s_n(\nu) \leq \sum_{\sigma \in V}2^{-s|\sigma|}.$$

\end{proposition}

\begin{proof}

We begin by showing that the desired statement holds in any $\omega$-model of $\WKL_0$.
Given our cover $\{N_\sigma\}_{\sigma \in V}$ of $E$, we may apply the ``closed set'' version of the Heine-Borel Covering Lemma (see~\cite{simpson2009subsystems}, IV.2) to choose a finite subcollection $N_{\sigma_0}, \dots , N_{\sigma_N}$ which still covers $E$. Define the open set $U = \bigcup_{i \leq N}N_{\sigma_i}$, which by assumption contains $E$ as a subset.

Now, define a sequence of trees $\langle T_k \rangle_{k \in \omega}$ by setting $\sigma \in T_k$ if and only if both of the following hold:
\begin{enumerate}
    \item $\sigma$ is compatible with some length-$k$ string $\tau$ for which $Z(\tau) = 1$ (i.e. $\tau$ is in the specified tree coding $E$)
    \item $\sigma$ does not extend any of the strings $\sigma_0, \sigma_1, \dots , \sigma_N$ represented in $U$
\end{enumerate}

It is straightforward to justify that each $T_k$ defined as such is closed under initial segments, so we do indeed get trees. Additionally, this sequence is computable from $Z$ - uniformly in $k$ - so exists in our model.

Consider the corresponding sequence $\langle C_k \rangle_{k \in \omega}$ of closed sets of $2^{\omega}$. We claim that the sequence is nested decreasing (i.e. $k < m \implies C_k \supseteq C_m$) and that the intersection $\bigcap_{k \in \omega} C_k$ is empty. The first statement follows from the fact that the trees $T_k$ themselves are nested decreasing: if a string $\sigma$ is compatible with a length-$k$ string $\tau$ in the tree for $E$, then $\sigma$ will also be compatible with all initial segments of $\tau$, each of which are also in this tree. (The second condition in the definition of $T_k$ does not depend on $k$.)

For the empty intersection, suppose to the contrary that there was some $X \in 2^{\omega}$ such that $X \in C_k$ for all $k \in \omega$. In particular, we would obtain that each initial segment $X \upharpoonright_k$ is in the corresponding tree $T_k$, which forces $Z(X\upharpoonright_k) = 1$; this implies that $X$ is in the closed set $E$. As $\{N_{\sigma_i}\}_{i \leq N}$ is a cover of $E$, there must exist some $\sigma_i$ for which $\sigma_i \subseteq X$. But then this string $\sigma_i$ is clearly not in any of the trees $T_k$, contradicting the assumption that $X$ was a path through each such tree.

At this point, we can used the nested closed set characterization of compactness, equivalent to $\WKL_0$ (see \cite{simpson2009subsystems}, IV.2), to conclude that some $C_k$ must be empty. In fact, all $C_k$ are empty for sufficiently large $k$, since the sets are nested. 

So choose $k \in \omega$ for which $C_k = \emptyset$ and for which $k \geq \max\{|\sigma_0|, \dots , |\sigma_N|\}$. Note that this implies $k \geq n$; since the original collection of cylinders $\{N_{\sigma}\}_{\sigma \in V}$ was a $2^{-n}$-cover, all $\sigma \in V$ (from which the strings $\sigma_i$ were all pulled from) must have length at least $n$.

We claim that $V^N = \{\sigma_0, \dots , \sigma_N\}$ will be considered in the definition for $\tilde{H}^s_n(Z^{\leq k})$, which must proceed according to Case 1 with $m = k$. Since all strings in $V^N$ have length between $n$ and $k$, condition (a) is satisfied. For condition (b), consider any length-$k$ string $\sigma$ for which $Z^{\leq k}(\sigma) = Z(\sigma) = 1$. Suppose we define $X \in 2^{\omega}$ by $X(i) = \sigma(i)$ if $i < |\sigma|$ and $X(i) = 0$ otherwise; such an element is computable so must exist in our model. If we examine the definition of the tree $T_k$, we see that every initial segment of $X$ will meet condition (1), using $\sigma$ as the witnessing length-$k$ string. However, since we assumed that $C_k$ was empty, we can conclude that some sufficiently long initial segment of $X$ must fail condition (2). In other words, $X$ must extend one of $\sigma_0, \dots , \sigma_N$. Since we assumed all of these strings had length at most $k$, it follows that $X \upharpoonright_k = \sigma$ extends the same string $\sigma_i \in V^N$, proving that $V^N$ meets condition (b).

Therefore, since $\tilde{H}^s_n(Z^{\leq k})$ computes the minimum $s$-weight over all valid sets $V$, we must have
\[
\tilde{H}^s_n(Z^{\leq k}) \leq W_s(V^N) = \sum_{i \leq N}2^{-s|\sigma_i|} \leq \sum_{\sigma \in V}2^{-s|\sigma|}.
\]
This completes one direction of the $\WKL_0$ equivalence. 

\medskip
For the opposite implication, assume we have an $\omega$-model for which the statement holds, and let $T \subseteq 2^{<\omega}$ be an infinite tree with associated code $Z \in 2^{\omega}$. Let $E \subseteq 2^{\omega}$ be the closed set associated to $T$. We wish to show that our model contains an element $X \in 2^{\omega}$ for which $X \in E = [T]$.

If this were not the case, then the empty set $\emptyset$ would be a valid $2^{-n}$-cover for the set $E$, whose associated ``weight'' (that is, the sum of all quantities $2^{-s|\sigma|}$ for $\sigma$ in the cover) would be zero. However, if $T$ is infinite, then for all initial segments $Z^{\leq k}$ for $k \geq n$, the quantity $\tilde{H}^s_{n}(Z^{\leq k})$ must be positive. This is certainly the case if $Z^{\leq k}$ falls under case (2) of the definition. If instead $Z^{\leq k}$ is sufficiently long to fall under case (1), then since there will always be at least one $\sigma \in T$ of length $k$, all covering sets $V$ under consideration will be nonempty. All these sets $V$ will then have positive values of $W_s(V)$, and therefore the minimum weight computed by $\tilde{H}^s_{n}(Z^{\leq k})$ is positive too. We can then conclude the same about all values of $\tilde{H}^s_n(\nu)$ by monotonicity.

This contradicts our statement when applied to the empty cover. Therefore, $E = [T]$ must be nonempty, proving that Weak König's Lemma is satisfied.
\end{proof}

From the above results, we can conclude that as long as we are working in a system at least as strong as $\WKL_0$, our two notions of Hausdorff $2^{-n}$-measure agree. 

\begin{corollary}\label{WKL-meas}
    The following holds in all $\omega$-models of $\WKL_0$: Let $E \subseteq \Cant$ be a standard closed set, and let $Z \in [S]$ be a code for a tree corresponding to $E$. Let $s \geq 0$ be a real number, and let $n \in \N$. Then if one of the quantities $H^s_n(Z)$ or $\mathcal{H}^s_{2^{-n}}(E)$ exists, the other quantity must also exist, and we will have 
    \[
    H_n^s(Z) =\mathcal{H}^s_{2^{-n}}(E).
    \]
\end{corollary}

\begin{proof}
Suppose we have an $\omega$-model of $\WKL_0$.
Assume first that $H^s_n(Z)$ exists, and denote it by $r \in \mathbb{R}$. Consider an arbitrary $2^{-n}$-cover $\{N_{\sigma}\}_{\sigma \in V}$ of $E$. If $\sum_{\sigma \in V}2^{-s|\sigma|} < r$, then by Proposition \ref{prop:WKL-meas1} we could find an initial segment $\nu \subseteq Z$ with 

$$\tilde{H}^s_n(\nu) \leq \sum_{\sigma \in V}2^{-s|\sigma|} < r,$$

contradicting $H^s_n(Z) = r$. So we conclude that $\sum_{\sigma \in V}2^{-s|\sigma|} \geq r$.

On the other hand, for $\epsilon > 0$, we can find some initial segment $\nu \subseteq Z$ with $\tilde{H}^s_n(\nu) < r + \epsilon$. We can apply Proposition \ref{prop:RCA-meas} to find a $2^{-n}$-cover $\{N_{\sigma}\}_{\sigma \in V}$ of $E$ with 

$$\sum_{\sigma \in V}2^{-s|\sigma|} \leq \tilde{H}^s_n(\nu) < r + \epsilon.$$

Combining these two deductions (and using the understanding of $\mathcal{H}^s_{2^{-n}}(E)$ discussed in the beginning of this section) yields that 

$$\mathcal{H}^s_{2^{-n}}(E) = r = H^s_n(Z).$$

The argument that the existence of $\mathcal{H}^s_{2^{-n}}(E)$ implies the existence and equality of $H^s_n(Z)$ is similar. 
\end{proof}

By this corollary, for $\omega$-models of $\WKL_0$ and stronger systems, we must also have
\[
\mathcal{H}^s(E) = \lim_{n \to \infty}\mathcal{H}^s_{2^{-n}}(E) = \lim_{n \to \infty} H^s_n(Z),
\]
assuming that all relevant quantities exist.

\medskip
If we are working in $\ACA_0$ or stronger, then for any tree code $Z \in \Cant$, the definition of $H^s_n(Z)$ can be fully formalized and within $ACA_0$ we can show that $H^s_n(Z)$ exists, since every bounded monotone sequence converges (see~\cite{simpson2009subsystems}, III.2). The same holds for $\lim_{n \to \infty} H^s_n(Z)$ if the value is finite, since the definition of $\tilde{H}^s_n$ is uniform in $n$. If the value is infinite, we can express this as an arithmetical formula. Therefore, in $\ACA_0$ we can work with $H^s(Z)$ as fully formalized.

\section{The Baire Category Theorem}

By using an appropriate topology on the space of closed sets (as represented by trees or pruned trees), one can view Besicovitch's construction of a subset of finite measure as a type of Baire Category argument\footnote{This was independently observed by Ted Slaman.}. In this section, we analyze the reverse mathematics strength of several Baire Category variants.

\subsection{Baire Category Theorem Formulations}
\label{ssec:bct}

The Baire Category Theorem of classical analysis already has a few equivalent forms, but one standard way it is often presented is as follows:

\begin{quote}
    \textit{``In any complete separable metric space, the intersection of countably many dense open sets is dense."}
\end{quote}

If we have a closed subset of a complete separable metric space, then the induced subspace is itself a complete separable metric space. Therefore, if we have a countable sequence of dense open sets in this subspace topology, the Baire Category Theorem applies. Using the subspace topology, we can formulate this corollary of the Baire Category theorem in the following way:

\begin{theorem}[Baire Category Theorem for Closed Sets, or BCTC]
    Let $X$ be a complete metric space, and let $F \subseteq X$ be a nonempty closed set. Suppose $\{U_n\}_{n \in \omega}$ is a sequence of nonempty open sets in $X$ such that each $U_n$ has the following property: if $N \subseteq X$ is a basic open set such that $N \cap F \neq \emptyset$, then there exists some $x \in N \cap F \cap U_n$. 
    Then $\bigcap_{n \in \omega}U_n$ also has this property: that is, if $N$ is a basic open set such that $N \cap F \neq \emptyset$, then there exists $x \in N \cap F \cap \left ( \bigcap_{n \in \omega} U_n\right )$.
\end{theorem}

The reverse mathematics strength of the theorem above depends on how the open and closed sets are coded. As discussed in the introduction, the effect of varying the representation of the sequence of open sets has been investigated in several papers. Here, we focus on the representation of the closed set $F$.

We restrict our attention to Cantor space. When we speak of an open set in $\Cant$, we assume that it is coded in the manner outlined in Section~\ref{ssec:closed}: i.e. as a subset of $\Str$ which is closed upward under taking extensions. We also introduce some terminology to make our theorems more concise.

\begin{definition} Let $F \subseteq \Cant$ be a nonempty closed set, and let $A \subseteq \Cant$ be any set. We say that $A$ is \textit{dense in} $F$ if for every string $\sigma \in \Str$ such that $N_{\sigma} \cap F \neq \emptyset$, there exists some $X \in N_{\sigma} \cap F \cap A$.
\end{definition}

We now formulate three different versions of BCTC, depending on how the closed set is coded.

\begin{itemize}[leftmargin=6em, labelsep=2ex, itemsep=1em ]
\item[($\BCTC$-I)]
    Let $F \subseteq \Cant$ be a standard closed set. Suppose $\{U_n\}_{n \in \N}$ is a sequence of nonempty open sets in $\Cant$ such that each $U_n$ is dense in $F$. Then $\bigcap_{n \in \N}U_n$ is dense in $F$. 

\item[($\BCTC$-II)]
    Let $F \subseteq \Cant$ be a separably closed set. Suppose $\{U_n\}_{n \in \N}$ is a sequence of nonempty open sets in $\Cant$ such that each $U_n$ is dense in $F$. Then $\bigcap_{n \in \N}U_n$ is dense in $F$. 

\item[($\BCTC$-III)]
    Let $F \subseteq \Cant$ be a pruned closed set. Suppose $\{U_n\}_{n \in \N}$ is a sequence of nonempty open sets in $\Cant$ such that each $U_n$ is dense in $F$. Then $\bigcap_{n \in \N}U_n$ is dense in $F$. 
\end{itemize}

We first consider the second and third formulations, since these are the versions which require the least axiomatic power to prove.

\begin{proposition} \label{BCTC-II/III}
    $\BCTC\text{-}\mathrm{II}$ and $\text{-}\mathrm{III}$ are provable in $\RCA_0$.
\end{proposition}

\begin{proof}
We first prove $\BCTC$-III, as it resembles the ``classical'' Baire Category Theorem proof most closely.
Since $F$ is pruned closed, let $T \subseteq \Str$ denote a pruned tree which represents $F$. Additionally, let $\langle V_n\rangle_{n \in \mathbb{N}}$ be a sequence of open codes corresponding to the sequence of open sets $\langle U_n\rangle_{n \in \mathbb{N}}$.

Suppose $\sigma^* \in \Str$ is such that $N_{\sigma^*} \cap F \neq \emptyset$; note that this implies  $\sigma^* \in T$.
We wish to construct an element $X \in N_{\sigma^*} \cap F \cap \left ( \bigcap_{n \in \mathbb{N}}U_n\right )$. We do this by defining a sequence of strings $\langle \sigma_n \rangle_{n \in \mathbb{N}}$ with the following properties:
\begin{enumerate}
    \item $\sigma^* \subsetneq \sigma_0$, and for all $n \in \mathbb{N}$ we have $\sigma_n \subsetneq \sigma_{n+1}$.
    \item For all $n \in \mathbb{N}$, we have $\sigma_n \in V_n$.
    \item For all $n \in \mathbb{N}$, we have $\sigma_n \in T$.
\end{enumerate}

If we have such a sequence, there exists a unique $X \in 2^{\mathbb{N}}$ ($\Delta^0_1$-definable in the sequence) which has all elements of the sequence as initial segments. For such an $X$, it is clear from properties (1) and (2) that $X \in N_{\sigma^*} \cap \left ( \bigcap_{n \in \mathbb{N}}U_n \right ).$ The fact that $X \in F$ follows from $X$ having arbitrarily long initial segments in $T$ by property (3).

We will define our sequence using a combination of primitive recursion and minimization in $\RCA_0$. Consider the set $A \subseteq \Str \times \mathbb{N} \times \Str$ defined by 
$$A = \{ (\sigma, n, \tau): \sigma \in T \rightarrow \tau \supsetneq \sigma \wedge \tau \in V_n \cap T\}.$$
Note that this set $A$ is $\Delta^0_1$-definable from the parameters $T$ and $\langle V_n \rangle_{n \in \mathbb{N}}$.

We claim that for every pair $(\sigma, n) \in \Str \times \mathbb{N}$, there exists some $\tau \in \Str$ for which $(\sigma, n, \tau) \in A$. If $\sigma \notin T$, then any element $\tau \in \Str$ will work as a witness. Otherwise, one can construct a sequence $X \in [T] = F$ by primitive recursion which extends $\sigma$, i.e. $N_\sigma \cap F$ is nonempty. Since the set $U_n$ is dense in $F$, there must exist some $Y \in N_\sigma \cap F \cap U_n$. We know all initial segments of $Y$ will be in $T$, and all sufficiently long initial segments of $Y$ will be in the open code $V_n$. Thus, we may choose such an initial segment $\tau$ which properly extends $\sigma$, and we obtain $(\sigma, n, \tau) \in A$.

We can now use minimization to define a function $g: \Str \times \mathbb{N} \to \Str$ by setting $g(\sigma, n)$ to be the least $\tau$ for which $(\sigma, n, \tau) \in A$. We use primitive recursion to define our desired sequence $\langle \sigma_n \rangle_{n \in \mathbb{N}}$ by 

$$\sigma_0 = g(\sigma^*, 0),$$
$$\sigma_{n+1} = g(\sigma_{n}, n+1),$$

for all $n \in \mathbb{N}$. By the definition of $g$, we will have $(\sigma^*, 0, \sigma_0) \in A$, as well as $(\sigma_n, n+1, \sigma_{n+1}) \in A$ for all $n \in \mathbb{N}$. Since $\sigma^* \in T$, it follows from a straightforward $\Sigma^0_1$-induction that $\sigma_n \in T$ for all $n \in \mathbb{N}$. The definition of $A$ then implies that each $\sigma_n$ satisfies $\sigma_n \supsetneq \sigma_{n-1}$ (or $\sigma_0 \supsetneq \sigma^*$, in the case of $n = 0$) and $\sigma_n \in V_n$. Thus, all of our specified properties are satisfied, so our desired element $X \in N_{\sigma^*} \cap F \cap (\bigcap_{n \in \mathbb{N}}U_n)$ exists.

\medskip
For $\BCTC$-II, we use the same notation for representing the open sets, but now we have a countable sequence $S = \langle Z_k \rangle_{k \in \mathbb{N}}$ representing the separably closed set $F$. However, the argument proceeds in a similar way. 

Once again, we let $\sigma^* \in \Str$ be such that $N_{\sigma^*} \cap F \neq \emptyset$. Observe that there must be some $k^* \in \mathbb{N}$ for which $\sigma^* \subseteq Z_{k^*}$, since any $Y \in N_{\sigma^*} \cap F$ must agree with some element of $S = \langle Z_k \rangle_{k \in \mathbb{N}}$ up through the length of $\sigma^*$.

We will again construct an element $X \in N_{\sigma^*} \cap F \cap \left ( \bigcap_{n \in \mathbb{N}}U_n\right )$ by defining a strictly increasing sequence of strings $\langle \sigma_n \rangle_{n \in \mathbb{N}}$. The desired properties below are almost the same as in the last proof, except for property (3):
\begin{enumerate}
    \item $\sigma^* \subsetneq \sigma_0$, and for all $n \in \mathbb{N}$ we have $\sigma_n \subsetneq \sigma_{n+1}$.
    \item For all $n \in \mathbb{N}$, we have $\sigma_n \in V_n$.
    \item For all $n \in \mathbb{N}$, there exists some $k \in \mathbb{N}$ such that $\sigma_n \subseteq Z_k$.
\end{enumerate}

If we define $X \in 2^{\mathbb{N}}$ to be the unique limiting element of this sequence, then it is again clear from properties (1) and (2) that $X \in N_{\sigma^*} \cap \left ( \bigcap_{n \in \mathbb{N}}U_n \right ).$ By property (3), (and the fact that each $\sigma_n$ satisfies $|\sigma_n| \geq n$), we can find arbitrarily long initial segments of $X$ which agree with some element of the sequence $S = \langle Z_k \rangle_{k \in \mathbb{N}}$, implying that $X \in \bar{S} = F.$

We again define our sequence using a combination of primitive recursion and minimization in $\RCA_0$, although we streamline a couple of steps. First, we claim that for every triple $(k, \sigma, n) \in \mathbb{N} \times \Str \times \mathbb{N}$, there exists some $(m, \tau) \in \mathbb{N} \times \Str$ for which the following condition holds:
\[
\sigma \subseteq Z_k \rightarrow \tau \supsetneq \sigma \wedge \tau \in V_n \wedge \tau \subseteq Z_m.
\]
If $\sigma \nsubseteq Z_k$, then any pair $(m, \tau)$ will work as a witness. Otherwise, we know $N_\sigma \cap F$ is nonempty, as witnessed by the element $Z_k$. Since the set $U_n$ is dense in $F$, there must exist some $Y \in N_\sigma \cap F \cap U_n$. We can then choose an initial segment $\tau \subseteq Y$ sufficiently long that $\tau \supsetneq \sigma$ and $\tau \in V_n$, and by the definition of $\bar{S}$, there must exist some $Z_m$ in the sequence $S$ such that $Z_m \supseteq \tau$. This pair $(m, \tau)$ will serve as our desired witness.

Since the condition we described is $\Delta^0_1$ from the parameters $\langle V_n \rangle_{n \in \mathbb{N}}$ and $S = \langle Z_k \rangle_{k \in \mathbb{N}}$, we can now use minimization to define a function $g: \mathbb{N} \times \Str \times \mathbb{N} \to \mathbb{N} \times \Str$ to be the least pair $(m, \tau)$ for which our condition holds. We then use primitive recursion to define a sequence $\langle (k_n, \sigma_n) \rangle_{n \in \mathbb{N}}$ by
\begin{align*}
(k_0, \sigma_0) & = g(k^*, \sigma^*, 0), \\
(k_{n+1}, \sigma_{n+1}) & = g(k_n, \sigma_n, n+1), 
\end{align*}
for all $n \in \mathbb{N}$. By the condition used to define $g$, we can determine after a simple $\Sigma^0_1$-induction that each pair $(k_n, \sigma_n)$ satisfies $\sigma_n \subseteq Z_{k_n}$. We then obtain that each $\sigma_n$ satisfies $\sigma_n \in V_n$ and $\sigma_{n} \supsetneq \sigma_{n-1}$ (or $\sigma_0 \supsetneq \sigma^*$, in the case of $n = 0$). Thus, if we consider the sequence $\langle \sigma_n \rangle_{n \in \mathbb{N}}$ - which can be obtained in a $\Delta^0_1$ way from the sequence $\langle (k_n, \sigma_n) \rangle_{n \in \mathbb{N}}$ just by applying the second projection map to each entry - we see that all of our specified properties are satisfied. Thus, our limiting element $X \in N_{\sigma^*} \cap F \cap (\bigcap_{n \in \mathbb{N}}U_n)$ exists, as desired.
\end{proof}

We finally turn to the complexity of $\BCTC$-I, which we will henceforth refer to as just $\BCTC$. One might start by attempting to carry out a proof in $\RCA_0$ by following the structure of the argument for $\BCTC$-III. However, the fact that the tree $T$ representing $F$ is not necessarily pruned causes issues with the inductive construction. 

Of course, if we have access to $\ACA_0$ we can just replace the tree $T$ with a pruned subtree representing the same closed set $F$, and the $\BCTC$-III proof goes through. But considering that other versions of the Baire Category theorem are provable in $\RCA_0$, one might conjecture that BCTC is still relatively ``simple" from a reverse math perspective, even if its power does reach beyond $\RCA_0$. However, it turns out that $\ACA_0$ is in fact necessary, as we will now show. 

\begin{definition} 
 Let $(.,.): \mathbb{N} \times \mathbb{N} \to \mathbb{N}$ denote a standard numerical pairing function. For $X \in 2^{\mathbb{N}}$ and $n \in \mathbb{N}$, we define the \textit{n-th column} $(X)_n \in 2^{\mathbb{N}}$ by

 $$(X)_n(m) = X((n,m))$$

 for all $m \in \mathbb{N}$. Similarly, if $\sigma \in \Str$ and $n \in \mathbb{N}$, we define the $n$-th column $(\sigma)_n$ to be the unique string $\tau$ such that $|\tau|$ is the least $m \in \mathbb{N}$ for which $(n,m) \geq |\sigma|$, and such that

 $$\tau(m) = \sigma((n,m))$$

 for all $m < |\tau|.$
\end{definition}

The following lemma is straightforward to verify.

\begin{lemma}
    The following are provable in $\RCA_0$ for the the column operators defined above:
    \begin{enumerate}
        \item [(i)]For any $X \in 2^{\mathbb{N}}$, the sequence of $n$-th columns $\langle (X)_n\rangle_{n \in \mathbb{N}}$ always exists. As a consequence, each individual column $(X)_n$ exists for a fixed $n \in \mathbb{N}$.
        \item [(ii)]For any $\sigma \in \Str$ and $n \in \mathbb{N}$, the $n$-th column $(\sigma)_n \in \Str$ is well-defined. Additionally, there exists a function $f: \mathbb{N} \times \Str \to \Str$ such that $f(n, \sigma) = (\sigma)_n$. 
        \item [(iii)] For any $n \in \mathbb{N}$, $\sigma, \tau \in 2^{\mathbb{N}}$, and $X \in 2^{\mathbb{N}}$, we have
        $$\sigma \subseteq \tau \implies (\sigma)_n \subseteq (\tau)_n,$$
        $$\sigma \subseteq X \implies (\sigma)_n \subseteq (X)_n.$$
    \end{enumerate}
\end{lemma}

\begin{theorem} \label{thm:BCTC}
    Over $\RCA_0$, $\BCTC$ is equivalent to $\ACA_0.$
\end{theorem}

\begin{proof}
Assume first that $\ACA_0$ holds. Then by Theorem \ref{ACA-T}, the standard closed set $F$ in the statement of $\BCTC$ is also pruned closed. We can then apply $\BCTC$-III, which does not require any axioms beyond $\RCA_0$.
    
\medskip
For the opposite implication, we use $\BCTC$ to prove that every one-to-one function on the natural numbers has a range. So let $f: \mathbb{N} \to \mathbb{N}$ be such a function. We split the proof into three steps.

\medskip
    \textit{Step 1:} Use $f$ to define a standard closed set $F$ and a sequence of open sets $\langle U_n \rangle_{n \in \mathbb{N}}$.

\medskip
    We define a tree $T \subseteq \Str$ by setting $\sigma \in T$ if and only if for all $n < |\sigma|,$ at least one of the following holds:
    \begin{enumerate}
        \item For all $i < |(\sigma)_n|,$ we have $(\sigma)_n(i) = 0$.
        \item There exists no $k < |\sigma|$ such that $f(k) = n$.
    \end{enumerate}
    Observe that $T$ is $\Delta^0_1$-definable from $f$, so its existence is provable in $\RCA_0.$
    Additionally, we see that $T$ is closed under initial segments, so we do indeed get a tree. For let $\sigma \in T$, let $\tau \subseteq \sigma$, and let $n < |\tau|$. As $n < |\sigma|$, we know that either condition (1) or (2) will hold with respect to $n$ and $\sigma$. In either case, it is easy to see that the same condition must hold with respect to $n$ and $\tau$, since $|\tau| \leq |\sigma|$ and $(\tau)_n \subseteq (\sigma)_n$ (see the previous lemma). As this $n$ was arbitrary, we obtain that $\tau \in T$ as desired.
    
    We now define a sequence of open codes $\langle V_n \rangle_{n \in \mathbb{N}}$ by setting $\sigma \in V_n$ if at least one of the conditions below holds:
    \begin{enumerate}
        \item There exists $i < |(\sigma)_n|$ such that $(\sigma)_n(i) = 1$.
        \item There exists $k < |\sigma|$ such that $f(k) = n$.
    \end{enumerate}
Again, this sequence is $\Delta^0_1$-definable in $f$ so exists under $\RCA_0$. It is also straightforward to check that each $V_n$ is closed upwards under taking extensions, so gives a valid open code; the reasoning is essentially the dual from that used in justifying that the set $T$ above was a tree.

Now, let $F = [T]$ denote the standard closed set corresponding to $T$, and let $\langle U_n \rangle_{n \in \mathbb{N}}$ denote the sequence of open sets corresponding to the open codes in $\langle V_n \rangle_{n \in \mathbb{N}}$. 

\medskip

\textit{Step 2:} Prove that each open set $U_n$ is dense in the closed set $F$.

\medskip

Suppose $\sigma \in \Str$ is such that $N_{\sigma} \cap F$ is nonempty, and let $X \in N_{\sigma} \cap F$. For a given $n \in \mathbb{N}$, we will use this element $X$ to exhibit an element $Y \in 2^{\mathbb{N}}$ such that $Y \in N_\sigma \cap F \cap U_n$.

We consider two cases: first, suppose that there exists $k \in \mathbb{N}$ such that $f(k) = n$. (Classically, this just means $n\in \text{ran}(f)$, although we avoid using this terminology since we have not yet shown that $\text{ran}(f)$ formally exists.) In this case, any initial segment $\tau \subseteq X$ for which $|\tau| > k$ will be an element of the code $V_n$ by condition (2), implying that $X \in U_n$. Since $X$ was already assumed to be in $N_\sigma \cap F$, $X$ itself serves as our desired witness, and we are done.

Suppose instead that there is no $k \in \mathbb{N}$ for which $f(k) = n$. In this case, let $m \in \mathbb{N}$ be the length of the $n$-th column $(\sigma)_n$. Define an element $Y \in 2^{\mathbb{N}}$ by setting $Y(i) = 1$ if $i = (n,m)$ and $Y(i) = X(i)$ otherwise. Observe that $Y$ is $\Delta^0_1$-definable from $X, n,$ and $m$, so its existence is provable in $\RCA_0$.
We claim that $Y \in N_\sigma \cap F \cap U_n$.

First, note that the only entry of $Y$ which potentially differs from that of $X$ is the entry $(n, m)$; by the definition of length for the $n$-th column $(\sigma)_n$, we have $(n,m) \geq |\sigma|$, which means for all $i < |\sigma|$ we will have $Y(i) = X(i) = \sigma(i)$. Since we assumed $X \in N_\sigma$, we obtain $Y \in N_\sigma$ as well.

Additionally, if we take an initial segment $\tau \subseteq Y$ for which $|\tau| > (n,m)$, then we will have $m < |(\tau)_n|$, with $(\tau)_n(m) = \tau((n,m)) = Y((n,m)) = 1$. Thus, $\tau$ will be in the open code $V_n$ by condition (1), so $Y \in U_n$.

Finally, we argue that $Y \in F$. Let $\tau \subseteq Y$ be an arbitrary initial segment, and let $s < |\tau|$. We wish to show that $s$ and $\tau$ satisfy at least one of the two conditions for entering the tree $T$. If $s = n$, then we know condition (2) will be satisfied since we assumed that there was no $k \in \mathbb{N}$ for which $f(k) = n$. 

If $s \neq n$, consider the initial segment $\tau' \subseteq X$ for which $|\tau'| = |\tau|$. Since $X$ was assumed to be an element of $F = [T]$, we know $\tau' \in T$ and thus $\tau'$ satisfies one of conditions (1) and (2) with respect to $s$. If $\tau'$ and $s$ satisfy condition (2), then $\tau$ and $s$ will meet this condition as well since $|\tau| = |\tau'|$. Otherwise, for any $i < |(\tau')_s| = |(\tau)_s|$ we must have
\[
(\tau')_s(i) = \tau'((s, i)) = X((s,i)) = 0.
\]

One standard property of the pairing function (see Section II.2 of \cite{simpson2009subsystems}) is that it is injective: therefore, if $s \neq n$ we must have $(s, i) \neq (n,m).$ As $(n,m)$ was the only entry of $Y$ that was potentially changed from $X$, we can conclude that
\[
(\tau)_s(i) = \tau((s,i)) = Y((s,i)) = X((s,i)) = 0.
\]

Thus, $\tau$ and $s$ meet condition (1) for $T$; as these values were arbitrary, we obtain that $Y \in F$ as desired.

 \medskip
\textit{Step 3:} Apply $\BCTC$, and use the witness to conclude the existence of $\text{ran}(f)$.

\medskip
Since we have demonstrated that the hypotheses hold, we can now apply $\BCTC$ to conclude that $\bigcap_{n \in \mathbb{N}}U_n$ is dense in $F$. Observe that the sequence $X_0 \in 2^{\mathbb{N}}$ defined by $X_0(i) = 0$ for all $i \in \mathbb{N}$ would be an element of $F$, since any initial segment would meet condition (1) of entering $T$, for any value of $n$. Thus, for any initial segment $\sigma \subseteq X_0$, the set $N_\sigma \cap F$ would be nonempty. Fix an arbitrary such $\sigma$ (for instance, we could just take $\sigma = \langle \rangle$), and use density to obtain an element $Z \in 2^{\mathbb{N}}$ with $Z \in N_{\sigma} \cap F \cap (\bigcap_{n \in \mathbb{N}}U_n).$ We claim that the following correspondence holds for all $n \in \mathbb{N}$:
\[
n \in \text{ran}(f) \iff \forall i \in \mathbb{N} \: (Z)_n(i) = Z((n,i)) = 0.
\]

If this correspondence is true, then the parameter $Z$ provides a $\Pi^0_1$-definition for $\text{ran}(f)$. Combined with the $\Sigma^0_1$-definition provided by $f$ itself (i.e. $n \in \text{ran}(f)$ if and only if $\exists k \in \mathbb{N} \: f(k) = n$), we obtain that $\text{ran}(f)$ is $\Delta^0_1$-definable so exists via comprehension in $\RCA_0$.

To prove this correspondence, we actually show the contrapositive, i.e.

$$n \notin \text{ran}(f) \iff \exists i \in \mathbb{N} \:(Z)_n(i) = 1.$$

So let $n \in \mathbb{N}$ be arbitrary, and assume first that $n \notin \text{ran}(f)$. By assumption, we have $Z \in U_n$, which means there is some initial segment $\sigma \subseteq Z$ for which $\sigma \in V_n$. If $n \notin \text{ran}(f)$, then $\sigma$ cannot satisfy condition (2) for entering $V_n$, so we conclude that $\sigma$ satisfies condition (1). As $(\sigma)_n \subseteq (Z)_n$, for the given witness $i < |(\sigma)_n|$ we have $(Z)_n(i) = (\sigma)_n(i) = 1$, so our desired implication holds.

For the other implication, suppose $n$ is such that for some $i \in \mathbb{N}$, we have $(Z)_n(i) = 1$. Assume for contradiction that $n$ is in the range of $f$, and let $k \in \mathbb{N}$ be such that $f(k) = n$. Choose an initial segment $\sigma \subseteq Z$ long enough that we have both $|\sigma| > k$ and $|\sigma| > (n,i)$. By another monotonicity property of the pairing function, this last condition would also imply $|\sigma| > n$, and by definition we have $i < |(\sigma)_n|$.
Then for this choice of $\sigma$, both conditions (1) and (2) for entering the tree $T$ would fail with respect to $n$. This contradicts the assumption that $Z \in F$, implying that all initial segments of $Z$ are in $T$. Thus, we may conclude that $n \notin \text{ran}(f)$, completing the proof of this characterization and of the entire theorem.   
\end{proof}

\subsection{Monotone Functions on Closed Sets}
\label{ssec:monotone}

In this section, we introduce a new context in which we can consider $\BCTC$ and its reverse mathematical strength. This context involves the behavior of certain monotone functions defined on closed sets. 

\begin{definition}
    Let $T \subseteq \Str$ be a tree, and let $\tilde{f}: T \to \mathbb{R}$ be a function. 
   We say that $\tilde{f}$ is \textit{monotone decreasing} if $\tilde{f}(\sigma) \geq \tilde{f}(\tau)$ whenever $\sigma \subseteq \tau$. Similarly, we say that $\tilde{f}$ is \textit{monotone increasing} if $\tilde{f}(\sigma) \leq \tilde{f}(\tau)$ whenever $\sigma \subseteq \tau$.            
\end{definition}

\begin{definition} \label{def:monotone_limit}
    Let $T \subseteq \Str$ be a tree, and let $F = [T]$ be the corresponding closed set. If $\tilde{f}: T \to \mathbb{R}$ is a function and $X \in F$, we let $f(X) \in \mathbb{R}$ denote the real number
    $$f(X) = \lim_{n \to \infty}\tilde{f}(X \upharpoonright_n),$$
    assuming the limit exists.
\end{definition}

Observe that if $\tilde{f}$ is a monotone increasing function which is bounded above, or a monotone decreasing function which is bounded below, then the limit in the definition above always exists ``classically." However, this limiting value may not exist in the model for all $X \in F$ if we are working in a system strictly weaker than $\ACA_0$.

That said, we will sometimes abuse notation in the following way. If $\tilde{f}$ is a monotone decreasing function on a tree $T$, $X$ is an element of $[T]$, and $c$ is a real number, then we may write $f(X) \geq c$ to mean that $\tilde{f}(X \upharpoonright_n) \geq c$ for all $n \in \mathbb{N}$, even if the limiting value for $f(X)$ does not necessarily exist. Similarly, writing $f(X) < c$ is understood to mean that there exists some $n \in \mathbb{N}$ for which $\tilde{f}(X \upharpoonright_n) < c$. We use similar notation with the opposite inequalities for a monotone increasing function. 

Additionally, for convenience we may sometimes make statements about the abstract ``function'' $f$ from $[T]$ to $\mathbb{R}$, even though this object does not formally exist in second order arithmetic and must be interpreted in terms of the intermediate function $\tilde{f}$. This is analogous to the relationship between a tree and the closed set it induces, or between an open code and the associated open set.

With this broadening of notation, the following definition is intuitive.

\begin{definition} \label{def:inf}
    Let $T \subseteq \Str$ be a tree with corresponding closed set $F = [T]$, let $\tilde{f}: T \to \mathbb{R}$ be a monotone decreasing function, and let $f: F \to  \mathbb{R}$ be the induced function. We say that a real number $\alpha$ is an \textit{infimum} of $f$ over $F$ if the following conditions hold:
    \begin{enumerate}
        \item For all $X \in F$, $f(X) \geq \alpha$.
        \item For all $\epsilon > 0$, there exists some $X \in F$ such that $f(X) < \alpha + \epsilon$.
    \end{enumerate}
    The notion of \textit{supremum} for a monotone increasing function is defined similarly.
\end{definition}

It turns out that the realization of infima (or suprema) of certain types of functions gives equivalent characterizations of $\ACA_0$, which can be realized by considering their similarities to $\BCTC$. 

We observe that the conditions mentioned above are not sufficient to guarantee that the infimum is realized, even in the ``full'' $\omega$-model.

\begin{proposition} There exists a tree $T \subseteq 2^{<\omega}$, a monotone decreasing function $\tilde{f}: T \to \mathbb{R}$, and a real number $\alpha$ such that $\alpha$ is an infimum of the induced function $f$ over $F = [T]$, but there is no $X \in F$ for which $f(X) = \alpha$.
\end{proposition}

\begin{proof}
 Let $T$ be all of $2^{<\omega}$ (meaning that $F$ is all of $2^{\omega}$). We define a monotone decreasing function $\tilde{f}: 2^{<\omega} \to \mathbb{R}$ by setting $\tilde{f}(\sigma) = 1$ if $\sigma$ consists entirely of 0's (or if $\sigma$ is the empty string), and otherwise setting $\tilde{f}(\sigma) = 2^{-k}$, where $k$ is the position of the first 1 in $\sigma$. 
\end{proof}

However, there are more restricted types of monotone decreasing functions that do realize their infimums. To describe them, we introduce some further terminology. For brevity, we assume going forward that $T, F, \tilde{f}, f,$ and $\alpha$ are defined as in Definition \ref{def:inf}.

\begin{definition}
    Let $X$ be an element of $[T] = F$. We say that $X$ is a \textit{limit point for} $\alpha$ if for every $\sigma \subseteq X$ and every $\epsilon > 0$, there exists $Y \in F$ such that $\sigma \subseteq Y$ and $f(Y) < \alpha + \epsilon$.
\end{definition}

In other words, if $X$ is a limit point, then we can find points which are arbitrarily close to $Y$ which get mapped arbitrarily close to $\alpha$. 

The following definition will be crucial in linking the existence of infima to the Baire Category Theorem.

\begin{definition} \label{def:f-dense}
    We say that $f$ is \textit{dense above} $\alpha$ if every $X \in F$ is a limit point for $\alpha$.
\end{definition}

We can now state the main principle of this subsection.

\begin{DMMin}[$\DMM$]
   If $f$ is dense above $\alpha$, then there exists $X \in F$ for which $f(X) = \alpha$.
\end{DMMin}

Given the ``density" properties displayed by the notion of limit point, it is not terribly surprising that the Dense Monotone Minimum principle is related to $\BCTC$. While $\BCTC$ indeed proves $\DMM$, we are also able to show that $\DMM$ is equivalent to a (potentially weaker) variant of $\BCTC$, which which call $\BCTC$-N (`N' for \emph{nested}).

\medskip\noindent
\textbf{Nested Baire Category Principle} ($\BCTC$-N). Let $F \subseteq 2^{\mathbb{N}}$ be a standard closed set. Suppose $\{U_n\}_{n \in \mathbb{N}}$ is a sequence of open sets in $2^{\mathbb{N}}$ such that each $U_n$ is dense in $F$, and such that for all $n,k \in \mathbb{N}$, $n \leq k$ implies $U_n \supseteq U_k$. Then $\bigcap_{n \in \mathbb{N}}U_n$ is dense in $F$.

\medskip
Clearly, $\BCTC$ implies $\BCTC$-N. It can be shown (see~\cite{Gruner:2026a}) that the two statements are $\omega$-equivalent. We have not been able to show that the two principles are equivalent over $\RCA_0$, the main issue being that it appears induction needs be applied to formulas of arithmetic complexity beyond $\Sigma^0_1/\Pi^0_1$.

\begin{theorem} \label{thm:DMMin}
    Over $\RCA_0$, $\DMM$ is equivalent to $\BCTC\text{-}\mathrm{N}$.
\end{theorem}

We need to introduce some notation which will be helpful for the subsequent proofs, as well as some of the constructions in Section~\ref{sec:regularity}.

\begin{definition} \label{def:m_order}
    Let $\alpha$ and $\beta$ be real numbers within second order arithmetic, and let $\alpha = \langle a_n \rangle_{n \in \mathbb{N}}$ and $\beta = \langle b_n \rangle_{n \in \mathbb{N}}$ denote their rational sequence representations. Given $k \in \mathbb{N}$, we write $\alpha <_k \beta$ if $b_k - a_k > 2^{-k+1}$. Otherwise, we write $\alpha \geq_k \beta$.
\end{definition}

By the definition of order on the real numbers discussed in Section II.4 of \cite{simpson2009subsystems}, we have that $\alpha < \beta$ if and only if there exists some $k \in \mathbb{N}$ such that $\alpha <_k \beta$.

\begin{proof}[Proof of Theorem~\ref{thm:DMMin}]
    We first show that $\BCTC$-N implies $\DMM$. Let $T$, $F$, $\tilde{f}$, $f$, and $\alpha$ be as defined in the statement of Dense Monotone Minimum. 
    We define a ($\Delta^0_1$-definable) sequence of open codes $\langle V_n \rangle_{n \in \mathbb{N}}$ by setting
    \[
    V_n = \{\sigma \in 2^{<\mathbb{N}}: \sigma \notin T \text{ or } \sigma \in T \wedge \exists k, m \leq |\sigma| \: \tilde{f}(\sigma \upharpoonright_k) <_m \alpha + 2^{-n}\}.
    \]

    Since $T$ is a tree, it follows that each $V_n$ is closed upward under taking extensions, and thus yields a valid open code. Additionally, we observe that the sets $V_n$ are nested decreasing. If $\sigma \in V_n$ for some $n$, then we must also have $\sigma \in V_i$ for all $i < n$, as $\alpha + 2^{-n} < \alpha + 2^{-i}.$ (Note that if $\alpha \in \mathbb{R}$ and $q$ is rational, we can obtain a rational sequence representation for $\alpha + q$ by just adding $q$ to each entry of the rational sequence for $\alpha$.)

    Let $\langle U_n \rangle_{n \in \mathbb{N}}$ denote the sequence of open sets corresponding to the open codes $\langle V_n \rangle_{n \in \mathbb{N}}$. Note that this sequence is nested decreasing by the remark above. Furthermore, we claim that each set $U_n$ is dense in $F$. 
    
    Indeed, suppose $\sigma \in 2^{<\mathbb{N}}$ is such that $N_{\sigma} \cap F \neq \emptyset$, and let $X$ be an element of $F$ extending $\sigma$. Since $f$ is dense above $\alpha$, we know that $X$ is a limit point for $\alpha$. Therefore, there must be some $Y \in F$ for which $\sigma \subseteq Y$ and $f(Y) < \alpha + 2^{-n}$. By the definition of limit, there must be some $k \in \mathbb{N}$ for which $\tilde{f}(Y \upharpoonright_k) < \alpha + 2^{-n}$. Let $m \in \mathbb{N}$ be such that $\tilde{f}(Y \upharpoonright_k) <_m \alpha + 2^{-n}$, and let $\tau \subseteq Y$ be an initial segment such that $|\tau| \geq k, m$. It follows that this initial segment $\tau$ is an element of $V_n$, so $Y \in U_n$. As we also have $Y \in N_{\sigma} \cap F$, we have witnessed our density claim. 
    
    We can now apply $\BCTC$-N to conclude that $\bigcap_{n \in \omega}U_n$ is dense in $F$. Since $F$ is nonempty (this is guaranteed by condition (2) in the definition of infimum), the intersection $\left (\bigcap_{n \in \omega}U_n \right)\cap F$ is nonempty as well. We claim that for any $Z \in 2^{\mathbb{N}}$ lying in this intersection, we must have $f(Z) = \alpha.$ 
    
    On one hand, $Z \in F$ so we must have $f(Z) \geq \alpha$ by the definition of infimum. On the other hand, given any $\epsilon > 0$, we may choose $n \in \mathbb{N}$ for which $2^{-n} < \epsilon,$ and by assumption we have $Z \in U_n$. If we take $\sigma \subseteq Z$ for which $\sigma \in V_n$, we see that we must have $k, m \leq |\sigma|$ for which $\tilde{f}(\sigma \upharpoonright_k) = \tilde{f}(Z \upharpoonright_k) <_m \alpha + 2^{-n}.$ By definition of order, this gives $\tilde{f}(Z_k) < \alpha + 2^{-n} < \alpha + \epsilon$, and thus $f(Z) < \alpha + \epsilon$ by monotonicity. Since $\epsilon$ was arbitrary, we obtain $f(Z) = \alpha$ as desired.

\medskip
For the converse implication, assume that $F$ is a standard closed set, and that $\langle U_n\rangle_{n \in \mathbb{N}}$ is a sequence of nested decreasing open sets in $2^{\mathbb{N}}$ such that each $U_n$ is dense in $F$. We wish to show that the intersection $\bigcap_{n \in \omega}U_n$ is dense in $F$. 

So let $\nu \in 2^{<\mathbb{N}}$ be such that $N_{\nu} \cap F$ is nonempty. If $T \subseteq 2^{<\mathbb{N}}$ is our tree representing $F$, define a tree $T_{\nu}$ by 
$$T_{\nu} = \{ \sigma \in T: \sigma \subseteq \nu \text{ or } \nu \subseteq \sigma\}.$$

It is clear that $T_{\nu}$ is a tree, and that it is $\Delta^0_1$-definable from $T$ and $\nu$ so exists under $\text{RCA}_0$. We also see that the standard closed set $F_{\nu} \subseteq 2^{\mathbb{N}}$ corresponding to $T_\nu$ consists precisely of the elements in $N_\nu \cap F$.

Let $\langle V_n \rangle_{n \in \mathbb{N}}$ be the sequence of open codes representing the open sets $\langle U_n \rangle$. 
Define a monotone decreasing function
     $\tilde{f}: T_\nu \to \mathbb{R}$ by 
     \[
     \tilde{f}(\sigma) = 
    \begin{cases}
        1, \text{ if } \forall i < |\sigma|, \sigma \notin V_i \\
        2^{-k-1}, \text{ if } k < |\sigma|, \sigma \in V_k, \text{ and } \forall i < |\sigma|,  i > k \rightarrow \sigma \notin V_i\}. 
    \end{cases}
     \]
Intuitively, the above definition just says that $\tilde{f}(\sigma)$ outputs $2^{-k-1}$ for the maximum value of $k < |\sigma|$ for which $\sigma \in V_k$, with $\tilde{f}(\sigma) = 1$ if no such $k$ exists. We just express it in the above manner to make it clear that the function is $\Delta^0_1$-definable. Given that the sets $V_n$ are all closed upwards under taking extensions, it is straightforward to check that $\tilde{f}$ is monotone decreasing. 

We claim that $0$ is an infimum of the induced function $f$ over $F_{\nu}$, and that every $X \in F_{\nu}$ is a limit point over $0$. On one hand, $0$ is certainly a lower bound for $f$, since $\tilde{f}$ only outputs positive values. We claim that for every $X \in F_{\nu}$, every $\sigma \subseteq X$, and every $\epsilon > 0$, there exists $Y \supseteq \sigma$ such that $f(Y) < \epsilon$. This will show both the claim regarding limit points, and the fact that $0$ is the greatest lower bound for $f$. Note that we assumed $N_\nu \cap F = F_\nu$ was nonempty, so we have at least one $X \in F_{\nu}$ over which we can apply this statement.

Indeed, given this arbitrary $X \in F_{\nu} = N_\nu \cap F$ and $\sigma \subseteq X$, let $\mu$ denote whichever of $\nu$ or $\sigma$ has maximal length; then $N_\mu \cap F = N_\sigma \cap F_\nu$, and this set is nonempty since it contains $X$. For $\epsilon > 0$, let $n \in \mathbb{N}$ be such that $2^{-n-1} < \epsilon$. By assumption, the open set $U_n$ is dense in our original closed set $F$, which means there exists some $Y \in N_\mu \cap F \cap U_n$. As $Y \in F_{\nu}$ and $Y \supseteq \sigma$, this $Y$ will serve as our desired witness so long as we show that $f(Y) < \epsilon$. To that end, take $\tau \subseteq Y$ sufficiently long that we have both $\tau \in V_n$ and $|\tau| > n$. 
Then by the definition of $\tilde{f}$, we will have $\tilde{f}(\tau) = 2^{-k-1}$ for some $k \geq n$.  This implies $\tilde{f}(\tau) \leq 2^{-n-1} < \epsilon,$ and thus $f(Y) < \epsilon$ as desired.

At this point, we have shown that the hypotheses of $\DMM$ hold, so we can apply this principle to choose an element $Z \in F_{\nu}$ for which $f(Z) = 0$. We claim that this element $Z$ is in $N_{\nu} \cap F \cap (\bigcap_{n \in \mathbb{N}}U_n)$, which would prove our density claim. The fact that $F_{\nu} = N_{\nu} \cap F$ means that the first two memberships are clear, so it remains to show that $Z \in U_n$ for all $n \in \mathbb{N}$. 

So given an arbitrary $n \in \mathbb{N}$, choose $\tau \subseteq Z$ sufficiently long that $\tilde{f}(\tau) \leq 2^{-n-1}$, which is possible by the definition of infimum. Then there is some $k < |\tau|$ such that $\tilde{f}(\tau) = 2^{-k-1}$, implying $\tau \in V_k$ and thus $Z \in U_k$. The fact that $\tilde{f}(\tau) = 2^{-k-1} \leq 2^{-n-1}$ implies that $k \geq n$; as our sequence of open sets was nested decreasing, this implies $Z \in U_n$ as desired.
\end{proof}

Combining Theorem~\ref{thm:DMMin} with Theorem~\ref{thm:BCTC}, we obtain the following.

\begin{corollary} \label{cor:DMMin_ACA}
    $\DMM$ is provable in $\ACA_0$.
\end{corollary}

\section{Measure Regularity Results}
\label{sec:regularity}

The function $H^s_n$ function represents a natural example of a monotone decreasing function on Cantor space. Therefore, the results from the previous section can help us deduce the reverse mathematical strength of some measure regularity results. These same techniques will also aid us in the analysis of the Besicovitch Theorem for Hausdorff measure in the final section.

\subsection{Regularity of \texorpdfstring{$H^s_n$}{H\^s\_n}-measure}
\label{ssec:regularity}

In this section, we assume that we have some fixed $s > 0$ and $n \in \omega$, and we will work with the $H^s_n$ measure as defined in Section~\ref{ssec:Hmeas}. For these results, we assume that we start with some (standard or pruned) closed set $F$, and we seek to construct a closed subset $E \subseteq F$ with some specified measure. 
The complexity of this construction varies depending on how the closed sets are represented as well as on how precisely the measures are specified.

Note that if we consider the specific measure $H^1_{1}$, we obtain a version of the standard Lebesgue measure. This will be equivalent to the normal definition of Lebesgue measure in $\WKL_0$ and stronger systems, but in $\RCA_0$ it still provides a well-defined function on the space of ``tree codes." 
We will consider the general measure $H^s_n$ (for $s \geq 0$ and $n \in \omega$) here, since those versions are needed for the discussion on Besicovitch's Theorem later. 

As in Section~\ref{ssec:monotone}, we will abuse notation slightly in writing inequalities involving $H^s_n$, since the limiting value of $H^s_n(Z)$ itself may not exist for every tree code $Z$ without the axioms of $\ACA_0$. (See the discussion following Definition \ref{def:monotone_limit}.)

First, we demonstrate how closed sets which satisfy a desired measure bound can be coded topologically. 
Recall from Section~\ref{ssec:closed} that for a non-trivial standard closed set $F \subseteq 2^{\omega}$ with a tree representation $T_F \subseteq 2^{<\omega}$, $S_{T_F}$ and $\tilde{S}_{T_F}$ represent the trees whose infinite paths code standard subtrees and pruned subtrees of $T_F$, respectively. 

\begin{definition} \label{S^c_F} 
Let $Z_F \in 2^{\omega}$ be a code for $T_F$, and let $c$ be a real number such that $0 \leq c \leq H^s_n(Z_F)$. Then we define $\mathcal{S}^c_F \subseteq 2^{\omega}$ as the closed set generated by the following tree:
$$\{\nu \in S_{T_F}: \forall k,m \leq |\nu| \;\:\tilde{H^s_n}(\nu\restriction_k) \geq_m c \}.$$
\end{definition}

Here $\geq_m$ denotes the order from Definition~\ref{def:m_order}. 
Note that the monotonicity of the function $\tilde{H}^s_n$ guarantees that the set above is indeed a tree.
We define $\tilde{\mathcal{S}}^c_F$ similarly, using $\tilde{S}_{T_F}$ in place of $S_{T_F}$. 

Observe that $T_F^c$ and $\tilde{T}_F^c$ are computable from $T_F$ and $c$, and thus any $\omega$-model of $\text{RCA}_0$ containing both of these objects must contain $T_F^c$ and $\tilde{T}_F^c$ as well. 

The following lemma is a straightforward consequence from the definition of $H^s_n$ and of order on the real numbers; we state it anyway given its importance to the arguments going forward.

\begin{lemma} \label{lem:H^s_n_c}
    Let $Z$ be an element of $[S_{T_F}]$, which codes a standard closed subset $E_Z \subseteq F$. Then $Z \in \mathcal{S}_F^c$ if and only if $H^s_n(Z) \geq c$. In an $\omega$-model of $\WKL_0$, this is further equivalent to $\mathcal{H}^s_{2^{-n}}(E_Z) \geq c$. An analogous statement can be made for the elements of $[\tilde{S}_{T_F}]$, which code pruned closed subsets of $F$.
\end{lemma}

We now proceed to our main lemma.

\begin{lemma} \label{H^s_n-inf} The following holds in all $\omega$-models of $\RCA_0$: Let $F$ be a nontrivial closed set, and let $\mathcal{S}^c_F$ and $\tilde{\mathcal{S}}^c_F$ be defined as above. If $\mathcal{S}^c_F$ (respectively $\tilde{\mathcal{S}}^c_F$) is nonempty, we have the following properties for the restriction of the monotone decreasing function $H^s_n$ (using the concepts introduced in Definitions~\ref{def:inf} and~\ref{def:f-dense}, respectively):
\begin{itemize}
    \item $c$ is an infimum of $H^s_n$ over $\mathcal{S}^c_F$ (respectively $\tilde{\mathcal{S}}^c_F$)
    \item $H^s_n$ is dense above $c$ over $\mathcal{S}^c_F$ (respectively $\tilde{\mathcal{S}}^c_F$)
\end{itemize}
\end{lemma}

\begin{proof}
For ease of notation, we will use the set $\mathcal{S}^c_F$ for the following proof. The argument for $\tilde{\mathcal{S}}^c_F$ is the same, and we will point out the places where the distinction is relevant. 

Assume $\mathcal{S}_F^c$ is nonempty. As mentioned above, all elements $Z \in \mathcal{S}^c_F$ will satisfy $H^s_n(Z) \geq c$. Therefore, it suffices to show that given an arbitrary $Z \in \mathcal{S}^c_F$, any finite initial segment $\tau \subseteq Z$, and any $\epsilon > 0$, there exists some $X \in \mathcal{S}^c_F$ with $\tau \subseteq X$ and $c \leq H^s_n(X)< c + \epsilon.$ 

First, by extending $\tau$ along $Z$ if necessary, we may assume without loss of generality that $\tau$ has the form $Z^{\leq n'}$ for some $n' \geq n.$ Let $K$ denote the number of length-$n'$ strings $\sigma$ for which $Z(\sigma) = 1$, and choose $m \geq n'$ so large that we have both $2^{-sm} < \epsilon$ and $K \cdot 2^{-sm} < c + \epsilon$. Define an element $Y_0 \in 2^{\omega}$ as follows: 
\begin{enumerate}[(1)]
\item If $|\sigma| \leq n'$, $Y_0(\sigma) = Z(\sigma)$.
\item If $n' < |\sigma| \leq m$, set $Y_0(\sigma) = 1$ if there exists at least one length-$m$ extension $\sigma' \supseteq \sigma$ for which $Z(\sigma') = 1$, and there exists no lexicographically smaller extension of $\sigma \upharpoonright_{n'}$ of length $|\sigma|$ which has this property. Otherwise, set $Y_0(\sigma) = 1$.
\item If $|\sigma| > m$, then set $Y_0(\sigma) = 1$ if $Z(\sigma) = 1$ and $Y_0( \sigma \upharpoonright_m) = 1$ (according to the previous case). Otherwise set $Y_0(\sigma) = 0$.
\end{enumerate}

Note that $Y_0$ is computable from $Z$, so it must exist in our $\omega$-model of $\text{RCA}_0$.
Given the fact that $Z \in [S_{T_F}]$, it is straightforward to check that $Y_0$ represents a tree, and that this tree is a subtree of $T_F$. Additionally, if $Z$ represents a pruned tree (i.e. if $Z \in [\tilde{S}_{T_F}]$), then $Y_0$ will represent a pruned tree as well. Therefore, $Y_0$ will be an element of either $[S_{T_F}]$ or $[\tilde{S}_{T_F}]$ as needed.

Observe that by step (2) of our definition above, we will have at most $K$ length-$m$ strings $\sigma$ for which $Y_0(\sigma) = 1$, since each of the $K$-many length-$n'$ strings represented in $Z$ will either have no extensions represented in $Y_0$, or exactly one. If we take $V$ to consist of exactly these extensions, then $W_s(V) \leq K \cdot 2^{-sm} < c + \epsilon$ by our choice of $m$.  Thus, any initial segment $\nu \subseteq Y_0$ for which $\nu \supseteq Y_0^{\leq m}$ will satisfy $\tilde{H}^s_n(\nu) < c + \epsilon$, and thus we have $H^s_n(Y_0) < c + \epsilon$ when we pass to the limit. Notice also that $Y_0 \supseteq \tau = Z^{\leq n'}$, since $Y_0$ agrees with $Z$ for all $\sigma$ of length at most $n'$.

Suppose that $Y_0$ agrees with $Z$ for all strings $\sigma$ of length $m$. Then by step (3) of our definition, $Y_0$ must also agree with $Z$ for all strings of length greater than $m$. This implies that $\tilde{H}^s_n(Y_0^{\leq k}) = \tilde{H}^s_n(Z^{\leq k})$ for all $k \geq m$, and thus $H^s_n(Y_0) \geq c$ since $Z \in \mathcal{S}^c_F$. 
Therefore, $Y_0$ itself serves as our witnessing element $X \in \mathcal{S}^c_F$ with $H^s_n(X) < c + \epsilon,$ and we are done. 

Otherwise, there must be at least one length-$m$ string $\sigma$ for which $Y_0(\sigma) = 0$ but $Z(\sigma) = 1$. Let $\sigma_1, \dots , \sigma_t$ enumerate all of the length-$m$ strings with this property in lexicographic order. Inductively define elements $Y_1, \dots, Y_t$ in $2^{\omega}$ as follows:
\begin{enumerate}[(a)]
\item If $Y_{i-1}(\sigma) = 1$, set $Y_i(\sigma) = 1$.
\item If $Y_{i - 1}(\sigma) = 0$, set $Y_i(\sigma) = 1$ if $\sigma$ is compatible with $\sigma_i$ and $Z(\sigma) = 1$. Otherwise, set $Y_{i}(\sigma) = 0$.
\end{enumerate}

Again, one can check that these do indeed code subtrees of $T_F$ (pruned subtrees if necessary), that they each extend the string $\tau$, and that they must exist in any $\omega$-model of $\text{RCA}_0$. (In particular, each $Y_i$ is computable from $Z$ and $Y_{i-1}$.) 

Note also that for all strings $\sigma$ of length at least $m$, we will have $Y_t(\sigma) = Z(\sigma)$. This is because any such string $\sigma$ for which $Z(\sigma) = 1$ must either satisfy $Y_0(\sigma) = 1$, or else must extend one of $\sigma_1, \dots , \sigma_t$. Thus, we will have $\tilde{H}^s_n(Z^{\leq k}) = \tilde{H}^s_n(Y_t^{\leq k})$ for all $k \geq m$, implying that $H^s_n(Y_t) \geq c.$

We now claim that some $Y_i$ among $Y_0, Y_1, \dots, Y_t$ must satisfy $c \leq H^s_n(Y_i) < c + \epsilon.$ For consider any $k \geq m$, and consider the values of 
\[
\tilde{H}^s_n(Y_0^{\leq k}), \tilde{H}^s_n(Y_1 ^{\leq k}), \dots , \tilde{H}^s_n(Y_t ^{\leq k}).
\]

We already know that $\tilde{H}^s_n(Y_0 ^{\leq k}) < c + \epsilon$ and $\tilde{H}^s_n(Y_t ^{\leq k}) \geq c$ by previous remarks. Additionally, we observe that each $\tilde{H}^s_n(Y_i ^{\leq k})$ must satisfy 
\[
\tilde{H}^s_n(Y_{i-1} ^{\leq k}) \leq \tilde{H}^s_n(Y_i ^{\leq k}) < \tilde{H}^s_n(Y_{i-1} ^{\leq k}) + \epsilon.
\]

The first inequality follows from the fact that  $Y_{i-1}(\sigma) = 1$ implies $Y_{i}(\sigma) = 1$ for all $\sigma \in 2^{<\omega}$. Thus any valid cover for computing $\tilde{H}^s_n(Y_i ^{\leq k})$ is also valid for computing $\tilde{H}^s_n(Y_{i-1} ^{\leq k})$. For the second inequality, we observe that if $V$ is the minimal cover whose weight realizes $\tilde{H}^s_n(Y_{i-1} ^{\leq k})$, then adding $\sigma_i$ to this cover gives one which is valid when computing $\tilde{H}^s_n(Y_{i}^{\leq k})$, since the only strings $\sigma$ which differ between $Y_i$ and $Y_{i-1}$ are those compatible with $\sigma_i$. (Note also that $|\sigma_i| = m$ is in the appropriate range of $n$ to $k$.) By considering the weight of this new cover, we deduce that 

$$\tilde{H}^s_n(Y_{i} ^{\leq k}) \leq \tilde{H}^s_n(Y_{i-1} ^{\leq k}) + 2^{-sm} < \tilde{H}^s_n(Y_{i-1} ^{\leq k}) + \epsilon$$

by our choice of $m$.

It follows that as we enumerate through the values 
\[
\tilde{H}^s_n(Y_0 ^{\leq k}), \tilde{H}^s_n(Y_1 ^{\leq k}), \dots , \tilde{H}^s_n(Y_t ^{\leq k})
\]
we must at some point land in the range $[c, c + \epsilon)$, since we start at a value less than $c + \epsilon$, we end at a value at least $c$, and we cannot ``jump'' by more than $\epsilon$. Let $Y^{(k)}$ denote the first element $Y_i$ for which $\tilde{H}^s_n(Y_i ^{\leq k}) \in [c, c+\epsilon)$.
If we apply this analysis to all $k \geq m$, we conclude that some $Y_i$ must appear as $Y^{(k)}$ for infinitely many $k$ by the infinite pigeonhole principle, since we only have finitely many options for the identity of this element. 

For this choice of $Y_i$, we  certainly have $H^s_n(Y_i) < c + \epsilon$; any $k \geq m$ for which $Y^{(k)} = Y_i$ will work as a witness for $\tilde{H}^s_n(Y_i ^{\leq k}) < c + \epsilon$, at which point we can apply monotonicity and pass to the limit. But since we can find arbitrarily large values of $k$ for which $\tilde{H}^s_n(Y_i ^{\leq k}) \geq c$, it follows by monotonicity again that we must in fact have $\tilde{H}^s_n(\nu) \geq c$ for all $\nu \subseteq Y_i$, and thus $H^s_n(Y_i) \geq c$. This proves that $Y_i$ is our witnessing element of $\mathcal{S}_F^c$ (or $\tilde{\mathcal{S}}_F^c$), as desired.
\end{proof}

From the lemma above, we deduce that if we start with a closed set $F$ and we wish to find a subset $E$ of an approximate measure, then the construction can be carried out computably as long as $F$ and $E$ are of the same type (either both standard closed or both pruned closed).

\begin{proposition}
    The following holds in all $\omega$-models of $\RCA_0$: If $F \subseteq \Cant$ is a nontrivial standard closed (respectively, pruned closed) set with tree code $Z_F$, and $c$ is such that $0 \leq c \leq H^s_n(Z_F)$, then for any $\epsilon > 0$, there exists a nontrivial standard closed (respectively, pruned closed) set $E \subseteq F$ whose associated tree code $Z_E$ satisfies $c \leq H^s_n(Z_E) < c + \epsilon$.
\end{proposition}

\begin{proof}
If $H^s_n(Z_F) = 0$, then the result is immediate; we can only have $c = 0$ so we can just take $E = F$ regardless of the value of $\epsilon$. So assume that $H^s_n(Z_F) > 0$, and let $c$ be our given target measure satisfying $0 \leq c \leq H^s_n(Z_F).$ Let $\epsilon > 0$ be arbitrary.

Choose a real number $d > 0$ such that $d \geq c$, but such that we also have $d \leq H^s_n(Z_F)$ and $d < c + \epsilon$. If $c$ itself is positive, then we can just take $d = c$; otherwise, we can (for example) take $d = \frac{1}{2}\min\{H^s_n(Z_F), \epsilon\}.$ Since the tree code $Z_F$ satisfies $H^s_n(Z_F) \geq d$ for this choice of $d$, we know that $\mathcal{S}_F^d$ will be nonempty since it will at least contain the element $Z_F$. If $F$ happens to be pruned closed, then $\tilde{\mathcal{S}}_F^d$ will be nonempty for the same reason. Choose $\delta > 0$ sufficiently small that $d + \delta \leq c + \epsilon$.
By applying Lemma \ref{H^s_n-inf} and using the definition of infimum, we deduce that there exists an element $Z_E \in \mathcal{S}^d_F$ (or $\mathcal{S}^d_F$) which satisfies $H^s_n(Z_E) < d + \delta \leq c + \epsilon$. As $H^s_n(Z_E) \geq d \geq c$, and $Z_E$ must code an infinite tree since $d$ is positive, we have our desired result.
\end{proof}

If $F$ is standard closed but we want the subset $E$ to be pruned closed, then the analog to the above result is equivalent to $\WKL_0$ instead.

\begin{proposition} \label{WKL-reg}
    The following is $\omega$-equivalent to $\WKL_0$: 
    If $F \subseteq \Cant$ is a nontrivial standard closed set with tree code $Z_F$, and $c$ is such that $0 \leq c \leq H^s_n(Z)$, then for any $\epsilon > 0$, there exists a nontrivial pruned closed set $E \subseteq F$ whose associated tree code $Z_E$ satisfies $c \leq H^s_n(Z_E) < c + \epsilon$.
\end{proposition}

\begin{proof}
To show that $\WKL_0$ is sufficient to prove the statement, we note that since $H^s_n(Z_F) \geq c$, the tree corresponding to $\tilde{\mathcal{S}}^c_F$ must  be infinite. Indeed, consider the initial segment $Z_F^{\leq k}$ for any $k \in \omega$, and define a string $\nu^k$ of the same length by setting $\nu^k(\sigma) = 1$ if and only if $Z_F^{\leq k}(\sigma) = 1$, and for all lengths $\ell$ with $|\sigma| < \ell \leq k$ there exists at least one $\tau \supseteq \sigma$ with $|\tau| = \ell$ and $Z_F^{\leq k}(\tau) = 1$. Note that this string $\nu^k$ will meet the conditions for being in $\tilde{S}_{T_F}.$

We also claim that $\tilde{H}^s_n(\nu^k) \geq c$. If we take a cover $V \subseteq 2^{<\omega}$ whose weight realizes $\tilde{H}^s_n(\nu^k)$, then the corresponding union of cylinders would form a $2^{-n}$-cover for the set $F$. (Note that every infinite sequence $X \in F$ will have its length-$k$ initial segment $\sigma$ satisfying $Z^{\leq k}_F(\sigma) = \nu^k(\sigma) = 1$.) By Proposition \ref{prop:WKL-meas1} (which is provable in $\WKL_0$), there must be some initial segment $\nu' \subseteq Z_F$ for which  

$$ \tilde{H}^s_n(\nu') \leq \sum_{\sigma \in V}2^{-s|\sigma|} = W_s(V) = \tilde{H}^s_n(\nu^k).$$

Thus, if $\tilde{H}^s_n(\nu^k)$ were strictly less than $c$, we would contradict the assumption that $H^s_n(Z_F) \geq c$.

Therefore, each of the infinitely many sequences $\nu^k$ will be in the tree for $\tilde{\mathcal{S}}^c_F$. By $\WKL_0$, this tree must have at least one infinite path, so Lemma~\ref{H^s_n-inf} applies and allows us to find a code $Z_E$ for a nontrivial pruned closed set $E \subseteq F$ which satisfies $c \leq H^s_n(Z_E) < c + \epsilon,$ for any $\epsilon > 0$.

For the reversal, consider any infinite tree $T$, and let $Z \in 2^{\omega}$ and $F = [T]$ be the corresponding code and closed set respectively. Note that $F$ is nontrivial since $T$ is infinite. Choose any $c \geq 0$ such that $0 \leq c \leq H^s_n(Z_F)$, any $\epsilon > 0$, and apply the statement to choose a nontrivial pruned closed $E \subseteq F$ whose code $Z_E$ satisfies $c \leq H^s_n(Z_E) < c + \epsilon$. As $E$ is represented by an infinite pruned tree, we can just take its leftmost path $X$, which can be defined in a computable way. Since $E \subseteq F = [T]$, we have $X \in [T]$ which proves $\WKL_0$.
\end{proof}

In fact, we can actually obtain an equivalence with $\WKL_0$ just by requiring the existence of a pruned closed subset with exactly the same measure as the original, under the constraint that this original measure is actually defined in our model.

\begin{corollary} \label{WKL-exists}
    The following is $\omega$-equivalent to $\WKL_0$: 
    If $F \subseteq \Cant$ is a nontrivial standard closed set with tree code $Z_F$, and there exists a real number $d$ such that $H^s_n(Z_F) = d$, then there exists a nontrivial pruned closed set $E \subseteq F$ whose associated tree code $Z_E$ satisfies $H^s_n(Z_E) = H^s_n(Z_F) = d$.
\end{corollary}

\begin{proof}
    If $\WKL_0$ holds, then we can choose any $\epsilon > 0$ and apply the previous corollary to $c = d$ to obtain a nontrivial pruned closed set $E \subseteq F$ whose code $Z_E$ satisfies $d \leq H^s_n(Z_E) <d + \epsilon.$

     We can also use the equivalence of $H^s_n$ and $\mathcal{H}^s_{2^{-n}}$ in $\WKL_0$ (see Corollary \ref{WKL-meas}) to conclude that $\mathcal{H}^s_{2^{-n}}(F) = H^s_n(Z_F) = d$. This means that given any $\delta > 0$, there is some valid $2^{-n}$-cover of $F$ with weight less than $d + \delta$. As $E$ is a subset of $F$, this cover would be valid for $E$ as well, so we can conclude that $H^s_n(Z_E) < d + \delta$ as a consequence of Proposition \ref{prop:WKL-meas1}. As $\delta > 0$ was arbitrary, we can combine this observation with the previous inequality of $H^s_n(Z_F) \geq d$ to conclude that $H^s_n(Z_E) = d$ as desired.

The reversal is the same as for the previous proposition, again making use of the fact that we can computably find an infinite path through any infinite pruned tree.
\end{proof}

We now turn our attention to the assertion that we can obtain a closed set with an \textit{exact} positive measure, although unlike the previous corollary, this measure may be strictly less than that of the original set. We again break this down into cases depending on the representations of the sets involved. Consider the following statements:

\begin{itemize}[leftmargin=5em]
    \item[($\mathsf{Reg}$-S)] {\em If $F \subseteq \Cant$ is a nontrivial standard closed set with tree code $Z_F$, and $c$ is such that $0 \leq c \leq H^s_n(Z_F)$, then there exists a nontrivial standard closed set $E \subseteq F$ whose associated tree code $Z_E$ satisfies $H^s_n(Z_E) = c$.}
    \item[($\mathsf{Reg}$-SP)] {\em If $F \subseteq \Cant$ is a nontrivial standard closed set with tree code $Z_F$, and $c$ is such that $0 \leq c \leq H^s_n(Z_F)$, then there exists a nontrivial pruned closed set $E \subseteq F$ whose associated tree code $Z_E$ satisfies $H^s_n(Z_E) = c$.}
    \item[($\mathsf{Reg}$-P)] {\em If $F \subseteq \Cant$ is a nontrivial pruned closed set with tree code $Z_F$, and $c$ is such that $0 \leq c \leq H^s_n(Z_F)$, then there exists a nontrivial pruned closed set $E \subseteq F$ whose associated tree code $Z_E$ satisfies $H^s_n(Z_E) = c$.} 
    \item[($\mathsf{Reg}$-PS)] {\em If $F \subseteq \Cant$ is a nontrivial pruned closed set with tree code $Z_F$, and $c$ is such that $0 \leq c \leq H^s_n(Z_F)$, then there exists a nontrivial standard closed set $E \subseteq F$ whose associated tree code $Z_E$ satisfies $H^s_n(Z_E) = c$.}
\end{itemize}

We first verify that $\ACA_0$ is sufficient to prove all of these statements for $\omega$-models. 

\begin{proposition} Each of $\mathsf{Reg}\text{-}\mathrm{S}$, $\mathsf{Reg}\text{-}\mathrm{SP}$, $\mathsf{Reg}\text{-}\mathrm{P}$, and $\mathsf{Reg}\text{-}\mathrm{PS}$ is provable in $\ACA_0.$
\end{proposition}

\begin{proof}
 By Theorem \ref{ACA-T}, standard and pruned closed sets are equivalent over $\text{ACA}_0$. It follows that each of $\mathsf{Reg}$-S, $\mathsf{Reg}$-SP, $\mathsf{Reg}$-P, and $\mathsf{Reg}$-PS are pairwise equivalent under $\ACA_0$. We choose to show that $\mathsf{Reg}$-S is satisfied.

To that end, we use the fact that $\ACA_0$ proves the Dense Monotone Minimum principle ($\DMM$) (Corollary~ \ref{cor:DMMin_ACA}). Let $F$ and $c$ be as in the statement of $\mathsf{Reg}$-S. By Lemma \ref{H^s_n-inf}, the function $H^s_n$ is dense above its infimum $c$ over the nonempty closed set $\mathcal{S}^c_F$. Therefore, by DMMin there must be some $Z \in \mathcal{S}^c_F$ for which $H^s_n(Z) = c$. The corresponding closed set $E_Z \subseteq F \subseteq 2^{\omega}$ serves as our desired witness.
\end{proof}

Obtaining lower bounds for the complexity of these statements is more difficult. 

We will need the well-known fact that a computable tree whose corresponding closed set has computable, positive Lebesgue measure must have a computable path. We state it here adapted to our formalism. For a proof see~\cite{Gruner:2026a}.

\begin{lemma} \label{lem:path-from-comp-meas}
    The following holds in any $\omega$-model of $\RCA_0$: Let $E \subseteq 2^{\omega}$ be a nontrivial standard closed set, and let $Z_E$ be a code for its tree. Suppose there exists a real number $c > 0$ such that $H^1_0(Z_E) = c$. Then there must exist an element $X \in E$. 
\end{lemma}

We will now prove lower bounds for the complexity of $\mathsf{Reg}$-SP, $\mathsf{Reg}$-S, and $\mathsf{Reg}$-P.

\begin{proposition} \label{reg-rev}
    The following hold:
    \begin{enumerate}[(i)]
        \item  $\WKL_0 \leq_\omega \mathsf{Reg}\text{-}\mathrm{SP}$
        \item $\WWKL_0 \leq_\omega \mathsf{Reg}\text{-}\mathrm{S}$
        \item $\WWKL_0 \leq_\omega \mathsf{Reg}\text{-}\mathrm{P}$
    \end{enumerate}
\end{proposition}

\begin{proof} 
(i): Observe that $\mathsf{Reg}$-SP is simply a strengthening of the statement in Proposition \ref{WKL-reg}, which we have already shown is equivalent to $\WKL_0$. $\square$

\medskip
(ii): Let $T \subseteq 2^{<\omega}$ be a tree satisfying the hypothesis of $\WWKL_0$, i.e. a tree with the property that
\[
\lim_{k \to \infty}2^{-k} \cdot |\{\sigma \in T: |\sigma| = k\}| > 0.
\]
(Technically, we should understand the above to mean that the terms in our sequence are uniformly bounded away from zero, since formally the limit may not exist without $\ACA_0.$) 

If $Z \in 2^{\omega}$ is the code for $T$, then the above condition on $T$ implies that $H^1_0(Z) > 0$, using the argument used in the proof of Lemma~\ref{lem:path-from-comp-meas}. We can therefore choose $c > 0$ such that $H^1_0(Z) \geq c$ and apply Lemma~\ref{lem:path-from-comp-meas}.

\medskip
(iii): Again, let $T \subseteq 2^{<\omega}$ be an infinite tree satisfying the hypothesis of $\WWKL_0$. We will define a new tree $T' \subseteq 2^{<\omega}$ by setting $\sigma \in T'$ if and only if $\sigma$ satisfies one of the following conditions:
\begin{enumerate}[(1)]
    \item $\sigma \in T$
    \item There exists $\tau \subseteq \sigma$ such that $\tau \in T$, neither of $\tau^{\frown}0$ or $\tau^{\frown}1$ are in $T$, and $\sigma(i) = 0$ for all $i$ such that $|\tau| \leq i < |\sigma|$.
\end{enumerate}

$T'$ is a pruned tree containing our original tree $T$ and hence also satisfies the hypothesis of $\WWKL_0.$ In particular, if $Z'$ represents the code for $T'$, then we may choose $c > 0$ such that $H^1_0(Z') \geq c$. By $\mathsf{Reg}$-P, there exists a nontrivial pruned closed set $E' \subseteq [T']$ whose associated tree code $Z_E'$ satisfies $H^1_0(Z_E') = c$. 

We define a tree code $Z_E \in 2^{\omega}$ by setting $Z_E(\sigma) = 1$ if and only if we have both $Z_E'(\sigma) = 1$ and $\sigma \in T$. We claim that $H^1_0(Z_E) = c$. 
It suffices to show that $H^1_0(Z_E) \geq c$. 

Suppose for contradiction that $H^1_0(Z_E) < c$. This would imply that there was some $m \in \omega$ for which 
\[
\tilde{H}^1_0(Z_E^{\leq m}) =  2^{-m} \cdot |\{\sigma \in 2^{<\omega}: Z_E(\sigma) = 1 \wedge |\sigma| = m\}| < c.
\]
For this choice of $m$, suppose we have a length-$m$ string $\sigma$ for which $Z_E'(\sigma) = 1$ but $Z_E(\sigma) = 0$. Since $Z_E'$ represents a pruned tree, this $\sigma$ must have some infinite extension $X \in E' \subseteq [T'].$ However, since $Z_E(\sigma) = 0$, it follows that $\sigma \notin T$. From the definition of $T'$, this infinite sequence $X$ can only have one form: the longest initial segment of $\sigma$ in $T$, followed by an infinite string of 0's. In particular, for any $k \geq m$, there will be exactly one length-$k$ extension of $\sigma$ in the tree for $Z_E';$ namely, the length-$k$ initial segment of this sequence $X$.

Now, let $N_m$ denote the number of length-$m$ strings $\sigma$ which satisfy the property above, choose $\epsilon > 0$ so small that $\tilde{H}^1_0(Z_E^{\leq m}) < c - \epsilon$, and choose $k \geq m$ large enough that $N_m \cdot 2^{-k} < \epsilon.$ For this $k$, let $V_k'$ denote the set of length-$k$ strings $\sigma$ with $Z_E'(\sigma) = 1$. By previous considerations, we have
\[
\tilde{H}^1_0(Z_E'^{\leq k}) = W^1(V_k') \geq c,
\]
where the last inequality follows from our assumption on $Z_E'$ from $\mathsf{Reg}$-P. But now suppose we define
\[
V_{k,E} = \{ \sigma \in V_k': Z_E(\sigma \upharpoonright_m) = 1 \}.
\]
Then we will have 
\begin{align*}
W^1(V_{k,E}) &= 2^{-k} \cdot |\{ \sigma \in V_k': Z_E(\sigma \upharpoonright_m) = 1 \}| \\
&\leq 2^{-m} \cdot |\{\sigma \in 2^{<\omega}: Z_E(\sigma) = 1 \wedge |\sigma| = m\}| \\
&= \tilde{H}^1_0(Z_E^{\leq m}) \\
&< c - \epsilon.\\
\end{align*}

But now consider any $\sigma \in V_k' \setminus V_{k, E}$. In that case, we will have $Z_E'(\sigma \upharpoonright_m) = 1$ but $Z_E(\sigma \upharpoonright_m) = 0$, and by our previous discussion, there will be exactly $N_m$ strings $\sigma$ satisfying this property. Thus, 
\[
W_1(V_k' \setminus V_{k,E}) = N_m \cdot 2^{-k} < \epsilon.
\]
This means that for the total weight of $V_k'$, we have
\[
W_1(V_k') = W_1(V_{k,E}) + W_1(V_k' \setminus V_{k,E}) < (c - \epsilon) + \epsilon = c,
\]
a contradiction.

Since $H^1_0(Z_E) = c$, we can apply Lemma~\ref{lem:path-from-comp-meas} again to conclude the existence of some $X \in 2^{\omega}$ which lies in the corresponding closed set $E$. From the definition of $Z_E$, we have $E \subseteq [T]$, so our original tree $T$ contains an infinite path, as desired. 
\end{proof}

Given that any pruned closed set is standard closed, it is clear that $\mathsf{Reg}$-SP implies both $\mathsf{Reg}$-S and $\mathsf{Reg}$-P, and likewise that both $\mathsf{Reg}$-S and $\mathsf{Reg}$-P imply $\mathsf{Reg}$-PS. Beyond that, it is unclear whether any of these implications can be reversed using only the axioms of $\RCA_0$. However, we can show that any differences between these results disappear in $\omega$-models of $\WKL_0$. Compare this with Section~\ref{ssec:Hmeas}, where it was shown that for full interchangeability between the two representations, we need $\ACA_0.$

\begin{corollary} 
For $\omega$-models of $\WKL_0$, the regularity results of $\mathsf{Reg}$-S, $\mathsf{Reg}$-SP, $\mathsf{Reg}$-P, and $\mathsf{Reg}$-PS are all pairwise equivalent.
\end{corollary}
\begin{proof}
By the discussion in the preceding paragraph, it suffices to prove that under $\WKL_0$, $\mathsf{Reg}$-PS (the weakest statement) implies $\mathsf{Reg}$-SP (the strongest statement). 

Therefore, assume we have a nontrivial standard closed set $F \subseteq 2^{\omega}$ with tree code $Z_F$, and let $c$ be such that $0 \leq c \leq H^s_n(Z_F).$ Consider any positive $\epsilon > 0$. By Proposition \ref{WKL-reg} in $\WKL_0$, we may choose a nontrivial pruned closed set $F' \subseteq F$ whose associated tree code $Z_F'$ satisfies $c \leq H^s_n(Z_F') < c + \epsilon$.

Now, use $\mathsf{Reg}$-PS to find a nontrivial standard closed set $E' \subseteq F'$ whose associated tree code $Z_E'$ satisfies $H^s_n(Z_E') = c$. Then use Corollary \ref{WKL-exists} in $\WKL_0$ to find a nontrivial pruned closed set $E \subseteq E'$ whose associated tree code $Z_E$ satisfies $H^s_n(Z_E) = H^s_n(Z_E') = c$. As $E \subseteq E' \subseteq F' \subseteq F$, this $E$ serves as our desired witness for $\mathsf{Reg}$-SP.  
\end{proof}

One may wonder if there any measure regularity results which are $\omega$-equivalent to $\ACA_0$. As discussed in previous sections, one difficulty with using measure theoretic results in reverse mathematics comes from the fact that in weaker subsystems of second order arithmetic, the quantity $H^s_n(Z)$ may not actually exist for all tree codes $Z$. In fact, by modifying Yu's result for open sets~\cite{yu1987measure}, one can show that
$\ACA_0$ is actually required to obtain a well-defined Lebesgue measure for all closed sets.

That said, we can still meaningfully compare the measure of a (standard or pruned) closed set with a potential subset if we adapt our notation to account for the fact that $H^s_n(Z)$ may not be defined. If we have two closed sets $E$ and $F$ with respective tree codes $Z_E$ and $Z_F$ respectively, we will write $H^s_n(Z_E) \leq H^s_n(Z_F)$ to mean the following:
\[
\forall \epsilon > 0\: \exists k \in \omega\: \forall m \in \omega \: \tilde{H}^s_n(Z_{E}^{\leq k}) < \tilde{H}^s_n(Z_F^{\leq m}) + \epsilon.
\]

It is straightforward to show that if both quantities $H^s_n(Z_E)$ and $H^s_n(Z_F)$ exist, then the above statement is equivalent to the interpretation of $H^s_n(Z_E) \leq H^s_n(Z_F)$ as a normal inequality of real numbers. Naturally, we will write $H^s_n(Z_E) = H^s_n(Z_F)$ to mean that we have both $H^s_n(Z_E) \leq H^s_n(Z_F)$ and $H^s_n(Z_F) \leq H^s_n(Z_E)$.

We then see that passing from a standard closed to a pruned closed subset of the same measure is $\omega$-equivalent to $\ACA_0$.

\begin{proposition} The following is $\omega$-equivalent to $\ACA_0$: If $F \subseteq \Cant$ is a nontrivial standard closed set with associated tree code $Z_F$, then there exists a nonempty pruned closed set $E \subseteq F$ whose associated tree code $Z_E$ satisfies $H^s_n(Z_E) = H^s_n(Z_F)$.
    
\end{proposition}

\begin{proof}
The fact that $\ACA_0$ proves the stated result is straightforward, since standard and pruned closed sets are equivalent. Therefore, given a nontrivial standard closed set $F \subseteq 2^{\omega}$ with tree code $Z_F$, we can simply take a code $Z_F'$ for a pruned tree representing the same set. Applying the equivalence of the $H^s_n$ and $\mathcal{H}^s_{2^{-n}}$ measures in $\ACA_0$, we can conclude that 

$$H_n^s(Z_F') = \mathcal{H}^s_{2^{-n}}(F) = H_n^s(Z_F).$$

\medskip
For the reverse direction, we will show that the given statement is sufficient to prove that every one-to-one function $f: \omega \to \omega$ has a range. 

Given such a function $f$, we define a code $Z_F$ for the same tree $T$ which was used in the proof of (iii) $\implies \ACA_0$ in Proposition \ref{ACA-T}. As a reminder, this tree $T$ consists of all $\sigma \in 2^{<\omega}$ such that either
 \begin{enumerate}[(1)]
     \item $\sigma$ consists entirely of 0's
     \item $\sigma \supseteq \tau_n$ for some $n \in \omega$, and $f(k) \neq n$ for all $k < |\sigma|$,
 \end{enumerate}
where $\tau_n$ represents the length-$(n+1)$ string consisting of $n$ 0's followed by a single 1. For the corresponding closed set $F$, we will apply the hypothesis  to obtain a nontrivial pruned closed set $E \subseteq F$ whose tree code $Z_E$ satisfies $H^1_0(Z_E) = H^1_0(Z_F)$. 

We claim that for all $n \in \omega$, 
\[
Z_E(\tau_n) = 1 \; \Leftrightarrow \; n \notin \text{ran}(f). \tag{$\ddagger$} \label{equ:ran_from_tree}
\] 
It follows that the range of $f$ is computable from $Z_E$, and thus $\text{ran}(f)$ must exist in our $\omega$-model. 

To see~\eqref{equ:ran_from_tree}, suppose first that $n \in \omega$ is such that $Z_E(\tau_n) = 1$. Since $Z_E$ represents a pruned tree, there must be some infinite sequence $X \supseteq \tau_n$ with $X \in E \subseteq F$. From this, we can deduce that $n \notin \text{ran}(f)$, since otherwise some sufficiently long initial segment of $X$ would be excluded from the tree coded by $Z_F$.

Suppose instead that $n$ is such that $Z_E(\tau_n) = 0$. Choose $\epsilon > 0$ small enough that $\epsilon < 2^{-|\tau_n|} = 2^{-n-1}$. Since $H^1_0(Z_E) = H^1_0(Z_F)$, there must be some $k \in \omega$ such that for all $m \in \omega$, we have
$$\tilde{H}^1_0(Z_F^{\leq k}) < \tilde{H}^1_0(Z_F^{\leq m}) + \epsilon.$$
By monotonicity, we may assume that $k \geq n+1$. Applying the above inequality to $m = k$, we can conclude
$$2^{-k} \cdot |\{\sigma \in 2^{<\omega}: Z_F(\sigma) = 1 \wedge |\sigma| = k\}| < 2^{-k} \cdot |\{\sigma \in 2^{<\omega}: Z_E(\sigma) = 1 \wedge |\sigma| = k\}| + \epsilon,$$
and subsequently
$$|\{\sigma \in 2^{<\omega}: Z_F(\sigma) = 1 \wedge |\sigma| = k\}| < |\{\sigma \in 2^{<\omega}: Z_E(\sigma) = 1 \wedge |\sigma| = k\}| + \epsilon \cdot 2^k.$$

Note that any string $\sigma$ for which $Z_E(\sigma) = 1$ must also satisfy $Z_F(\sigma) = 1$. Again, this follows from the fact that $Z_E$ represents a pruned tree, so any $\sigma$ in this tree has some infinite extension $X \in E \subseteq F$, whose initial segments are all in the tree coded by $Z_F$.
Thus, we can deduce from the above inequality that the number of length-$k$ strings $\sigma$ for which $Z_F(\sigma) = 1$ but $Z_E(\sigma) = 0$ is less than $\epsilon \cdot 2^{k} < 2^{k - n - 1}$ by choice of $\epsilon$.

But now observe that since $Z_E(\tau_n) = 0$, none of the $2^{k-n-1}$-many length-$k$ extensions of $\tau_n$ can be in the tree coded by $Z_E$. If all of these strings were in the tree coded by $Z_F$, we would contradict the bound outlined above. Therefore, at least one of these strings $\sigma$ must be excluded from $Z_F$, which by its definition means that there is some $i < |\sigma| = k$ for which $f(i) = n$. Thus, $n \in \text{ran}(f)$, as desired.
\end{proof}

Compare the result from this proposition to that of Corollary \ref{WKL-exists}, where the measure of the original closed set $F$ is assumed to already exist.  

As an aside, this pair of results is somewhat analogous to some prior results concerning maxima/minima for continuous functions on compact sets. If the supremum of the function's values on the set exists, as a real number, then the existence of a point in the set realizing that supremum is equivalent to $\WKL_0$; otherwise, the existence is equivalent to $\ACA_0$. See \cite{simpson2009subsystems} for more details.

\subsection{Besicovitch's Theorem}
\label{ssec:besicovitch}

In this final section, we will use the techniques and results developed thus far to analyze the strength of the following regularity result for the $s$-dimensional Hausdorff measure, first proven by Besicovitch in 1952~\cite{besicovitch1952existence}.

\begin{theorem}[Besicovitch]
Let $s > 0$, and let $F \subseteq 2^{\omega}$ be a closed set with $\mathcal{H}^s(F) = \infty$. Then for any real number $c > 0$, there exists a closed subset $E \subseteq F$ with $\mathcal{H}^s(E) = c$. 
\end{theorem}

Our main goal is to verify that this theorem is provable in $\ACA_0$. 
As discussed in the introduction, this fact is not obvious from Besicovitch's original argument. Fortunately, with the axioms of $\ACA_0$, the various ways of coding closed sets are all equivalent, as are the different ways of representing the $s$-dimensional Hausdorff $2^{-n}$-measure. Thus, the provability of Besicovitch's theorem in $\ACA_0$ does not depend on the precise way in which the theorem is formalized. We choose to prove the formulation below; this version is slightly stronger than the original statement, in that it does not require the original set's Hausdorff measure to be infinite, nor the witnessing subset's measure to be positive. 

\begin{BT}[$\BES$] \label{Bes}
    Let $F \subseteq \Cant$ be a nontrivial standard closed set with tree code $Z_F$, and suppose that $c \geq 0$, $s > 0$ are real numbers such that 
    \[
    c \leq \lim_{n \to \infty}H^s_{2^{-n}}(Z_F).
    \]
    Then there exists a nontrivial standard closed set $E \subseteq F$ whose associated tree code $Z_E$ satisfies
    \[
    \lim_{n \to \infty}H^s_n(Z_E) = c.
    \]
\end{BT}

We will derive this result as a consequence of $\BCTC$, which is equivalent to (and therefore provable in) $\ACA_0$ by Theorem \ref{thm:BCTC}. First, we introduce some notation. We assume throughout the remainder of this section that our parameter $s > 0$ and our given closed set $F \subseteq \Cant$ are fixed, as are the corresponding tree $T_F \subseteq \Str$ and tree code $Z_F \in \Cant$. 

\begin{definition} 
Let $c \geq 0$ be a real number, and let $n \in \N$. Define the tree $T^c_n$ by 
\[
T^c_n = \{ \nu \in S_{T_F}: \tilde{H}^s_n(\nu) \geq c\},
\]
and let $\mathcal{S}^c_n = [T^c_n]$ denote the corresponding closed set. In other words, $\mathcal{S}^c_n$ consists of all $Z \in [S_{T_F}]$ for which $H^s_n(Z) \geq c$. We will let $\mathcal{U}^c_n$ denote the open complement $2^{\N} \setminus \mathcal{S}^n_c$. 
\end{definition}

Note that the definition of $\mathcal{S}^c_n$ is essentially the same as that in Definition \ref{S^c_F}. However, in that section we assumed that the index $n$ was fixed, so we did not need to keep track of it. Here, we will be working with changing values of $n$ so we include that parameter as part of the notation. 

Our $\BCTC$ argument hinges on the following density property, which we will call \textit{Besicovitch-Density} or simply \textsf{BD}.

\begin{BD}[\textsf{BD}] \label{BD}  Let $n_0 \in \N$ be such that $\mathcal{S}^c_{n_0}$ is nonempty. If $d > c$, then for all $n \in \N$, the open set $\mathcal{U}^d_n$ is dense in the closed set $\mathcal{S}^c_{n_0}$. 
\end{BD}

If this proposition holds in $\ACA_0$, then if we choose a sufficiently large $n_0 \in \N$, and a decreasing sequence $(d_n)_{n \in \N}$ of real numbers converging downward to $c$, then $\BCTC$ will allow us to obtain an element $Z \in \mathcal{S}^c_{n_0} \cap \bigcap_{n \in \N}\mathcal{U}^{d_n}_n$. This element $Z$ will 
code a closed subset $E_Z \subseteq F$ satisfying $\mathcal{H}^s(E_Z) = c$. (We will provide the details at the end of this section.)

We present the proof of the Besicovitch Density property in $\ACA_0$ in a series of lemmas. First, we verify that the proposition holds in the case where the indices $n$ agree. 

\begin{lemma}[$\ACA_0$] \label{BD-fixed} Let $n \in \N$ be fixed, and assume $c \geq 0$ is such that $\mathcal{S}^c_{n}$ is nonempty. If $d > c$, then the open set $\mathcal{U}^d_n$ is dense in the closed set $\mathcal{S}^c_{n}$. 
\end{lemma}

\begin{proof}
 This result is essentially just a reformulation of Lemma \ref{H^s_n-inf}. It is straightforward to check that proof can be fully formalized in $\ACA_0$. Note that, unlike in $\RCA_0$, we have access to the infinitary pigeonhole principle in $\ACA_0$.
\end{proof}

Note that the previous lemma actually shows that $\mathcal{U}_k^d$ is dense in $\mathcal{S}_n^c$ for all $k \leq n$: For such a $k$, the closed set $E_Y$ corresponding to the witness $Y$ from the above proof will satisfy 
\[
\mathcal{H}^s_{2^{-k}}(E_Y) \leq \mathcal{H}^s_{2^{-n}}(E_Y) < d
\]
by the definition of the Hausdorff $\delta$-measure. Therefore, $Y \in \mathcal{U}^d_k$ as well.

\medskip
Returning to the proposition \textsf{BD}, by induction it now suffices to prove that if $\mathcal{U}_n^d$ is dense in $\mathcal{S}_{n_0}^c$ for some $n \geq n_0$, we also must have $\mathcal{U}_{n+1}^d$ dense in $\mathcal{S}_{n_0}^c$. Note that we must be somewhat careful, since $\ACA_0$ technically only allows us to apply induction over statements which are arithmetic. However, this does not present an issue, since as long as we have $\WKL_0$, the existence of an infinite path through a tree is equivalent to the existence of infinitely many strings in that tree. In particular, for any $\nu \in 2^{<\N}$ we can express the statement $N_{\nu} \cap \mathcal{S}^c_{n} \neq \emptyset$ as
$$\forall m \geq |\nu|, \exists \tau \supseteq \nu, \tau \in T^c_n \wedge |\tau| = m.$$

Likewise, the entire statement ``$\mathcal{U}^d_n$ \textit{is dense in} $\mathcal{S}^c_{n_0}$" can be expressed in an arithmetic way by noting that
    $$\forall \nu \in 2^{<\N} ( N_{\nu} \cap \mathcal{S}^c_{n_0} \neq \emptyset \implies \exists Y \in N_{\nu} \cap \mathcal{S}^c_{n_0} \cap \mathcal{U}^d_{n_0})$$
is equivalent to 
    $$ \forall \nu \in 2^{<\N}, (N_{\nu} \cap \mathcal{S}^c_{n_0} \neq \emptyset\implies \exists \tau \supseteq \nu, N_{\tau} \cap \mathcal{S}^c_{n_0} \neq \emptyset \wedge \tau \notin T^d_n).$$

Before we proceed to the main proof, let us briefly unpack the inductive hypothesis that $\mathcal{U}^d_n$ is dense in $\mathcal{S}^c_{n_0}$, where $n \geq n_0$.
 Suppose we have string $\tau \in 2^{<\N}$ such that  $N_{\tau} \cap \mathcal{S}_{n_0}^c \neq \emptyset$, and we have used this density assumption to obtain an element $Z \in N_{\tau} \cap \mathcal{S}_{n_0}^c \cap \mathcal{U}^d_n$. Consider an optimal covering $U$ witnessing that $\tilde{H}^s_n(Z^{\leq m}) < d$ for some $m \geq n$. By the definition of $\tilde{H}^s_n$ and the covers used there, we know that all strings in $U$ have length at least $n$. If all of these strings had length strictly greater than $n$ - that is, length at least $n + 1$ - then this same cover $U$ would be considered in the calculation of $\tilde{H}^s_{n+1}(Z^{\leq m})$, meaning that $Z$ would also be in $\mathcal{U}^d_{n+1}$ and we would be done.

Of course, this may not be the case. However, the next best option would be to know that, in a sense, some length-$n$ strings in the cover were unnecessary. One way to interpret this is that each length-$n$ string can be replaced by a set of proper extensions which still ``cover" the appropriate portion of $Z$ but which don't increase the original cover's weight - or at least, that we can ensure that these increases can be made arbitrarily small. 

This idea motivates the following definitions.

\begin{definition}
    Suppose $Z \in [S_{T_F}]$, and $\tau \in \Str.$ Let $Z_{\tau} \in [S_{T_F}]$ denote the sequence defined by $Z_{\tau}(\sigma) = 1$ iff $\sigma$ is compatible with $\tau$ and $Z(\sigma) = 1$. 
\end{definition}

In other words, $Z_\tau$ codes the closed set coded by $Z$ if we restrict the latter to $N_\tau$. Of course, this closed set may be empty.

\begin{definition}
    Let $Z \in [S_{T_F}]$, let $n \in \N$, and let $\tau \in \Str$ have length $n$. Define $Z$ to be \textit{n-thin for} $\tau$ if we have
    $$H^s_{n+1}(Z_{\tau}) \leq 2^{-ns}.$$
    If $Z$ is $n$-thin for all $\tau \in \Str$ with length $n$, then we simply say $Z$ is \textit{n-thin}.
    Let $\mathcal{T}_n$ denote the set of all $n$-thin sequences in $[S_{T_F}]$.
\end{definition}

As suggested by the informal discussion above, these $n$-thin sequences are sufficient to guarantee robustness when passing from $\mathcal{U}^d_n$ to $\mathcal{U}^d_{n+1}.$

\begin{lemma}[$\ACA_0$] \label{T-U_n}
 For all $n \in \N$ and $d > 0$, we have $\mathcal{U}_n^d \cap \mathcal{T}_n \subseteq \mathcal{U}^d_{n+1}.$
\end{lemma}

\begin{proof}
Let $Z \in \mathcal{U}_n^d \cap \mathcal{T}_n$. As $Z \in \mathcal{T}_n \subseteq [S_{T_F}]$, we know that $Z$ does represent a valid tree code, so the fact that $Z \in \mathcal{U}_n^d$ implies that $H^s_n(Z) < d$.
Let $U = \{\tau_0, \dots , \tau_{k-1}\}$ be an optimal cover for $Z$ witnessing that $\tilde{H}^s_n(Z^{\leq m}) < d$ for some $m \geq n$. 

If all strings in $U$ have length at least $n + 1$, then this cover also witnesses that $\tilde{H}^s_{n+1}(Z^{\leq m}) < d$, yielding $Z \in \mathcal{U}_{n+1}^d$ and we are done. 

Otherwise, assume without loss of generality that there is some $t$ with $1 \leq t \leq k$ for which $\tau_0, \tau_1, \dots , \tau_{t-1}$ are exactly the strings in $U$ of length $n$. Since $W_s(U) < d$, we may choose $\epsilon > 0$ small enough that we also have $W_s(U) < d - \epsilon.$ Since $Z$ is $n$-thin, we may choose $m'\geq m$ so large that for all length-$n$ $\tau_i$ in $U$, we have
$$\tilde{H}^s_{n+1}(Z_{\tau_i}^{\leq m'}) < 2^{-ns} + \tfrac{\epsilon}{t}.$$
For each $Z_{\tau_i}$, let $V_i \subseteq 2^{<\N}$ be a witnessing cover with $W_s(V_i) = \tilde{H}^s_{n+1}(Z_{\tau_i}^{\leq m'}),$
and let $U'$ be the cover obtained from $U$ by replacing each string $\tau_i$ with the strings from $V_i$. It is straightforward to verify that $U'$ is a valid cover for computing $\tilde{H}^s_{n+1}(Z^{\leq m'})$. Finally, we note that

\begin{align*}
    W_s(U') &= W_s(V_0) + \dots + W_s(V_{t-1}) + 2^{-s|\tau_{t}|} + \dots + 2^{-s|\tau_{k-1}|} \\
    &< (2^{-ns} + \tfrac{\epsilon}{t}) + \dots + (2^{-ns} + \tfrac{\epsilon}{t}) + 2^{-s|\tau_{t}|} + \dots + 2^{-s|\tau_{k-1}|} \\
    &= 2^{-s|\tau_{0}|} + \dots + 2^{-s|\tau_{t-1}|} + 2^{-s|\tau_{t}|} + \dots + 2^{-s|\tau_{k-1}|} + t(\tfrac{\epsilon}{t}) \\
    &= W_s(U) + \epsilon \\
    &< (d - \epsilon) + \epsilon \\
    &= d.
\end{align*}
This proves that $Z$ is an element of $\mathcal{U}_{n+1}^d$ as desired.
\end{proof}

All that remains to show now is that the $n$-thin elements are sufficiently numerous within each nonempty set $\mathcal{S}^d_{n_0}.$

\begin{lemma}[$\ACA_0$] \label{T-dense} Let $c \geq 0$ and let $n_0 \in \N$ be such that $\mathcal{S}^c_{n_0}$ is nonempty. Then for all $n \geq n_0$, the set $\mathcal{T}_n$ is dense in $\mathcal{S}^c_{n_0}$.
\end{lemma}

\begin{proof}
Let $N_\nu$ be such that $N_\nu \cap \mathcal{S}^c_{n_0} \neq \emptyset$; we wish to exhibit an element $Y \in N_\nu \cap \mathcal{S}^c_{n_0} \cap \mathcal{T}_n$. 

Choose an element $Z \in N_\nu \cap \mathcal{S}^c_{n_0}$. By extending $\nu$ along $Z$ if necessary, we may assume that $\nu$ is long enough for $\nu^{\leq n}$ to be defined. For each $\sigma \in 2^{<\N}$ of length $n$, let $\nu_{\sigma}$ denote the finite string with $|\nu_{\sigma}| = |\nu|$ and which has $\nu_{\sigma}(\tau) = 1$ if and only if $\nu(\tau) = 1$ and $\tau$ is compatible with $\sigma$. Note that for each such $\sigma$, we have $Z_{\sigma} \supseteq \nu_{\sigma}$.

\medskip
We split the proof into two main steps.

\medskip
\textit{Step 1:} Define an element $Y \in N_{\nu} \cap \mathcal{T}_n$.

\smallskip
We will define an element $Y \in [S_{T_F}]$ by specifying $Y_{\sigma}$ for each $\sigma \in 2^{<\N}$ of length $n$, and ensuring that each $Y_{\sigma}$ extends $\nu_\sigma$. In this way, the final sequence $Y$ will be well-defined and will extend $\nu$. If $\tau \in 2^{<\N}$ is in the domain of $\nu$, then $\tau$ is also in the domain of $\nu_\sigma$ for all length-$n$ strings $\sigma$. Thus for any length-$n$ $\sigma_i$ which is compatible with $\tau$ (either a prefix or extension, depending on the length of $\tau$), we will have

$$Y(\tau) = Y_{\sigma_i}(\tau) = \nu_{\sigma_i}(\tau) = \nu(\tau).$$

If $\tau \in 2^{<\N}$ is not in the domain of $\nu$, then we must have $|\tau| > n$, so the value of $Y(\tau)$ will be determined uniquely by $Y_{\tau \upharpoonright_n}(\tau)$.

It is straightforward to show that $Y$ will be an element of $[S_{T_F}]$ as long as each $Y_{\sigma}$ is. Additionally, as long as $\ACA_0$ can prove the existence of each $Y_\sigma$, we can also deduce the existence of the overall element $Y$, which is $\Delta^0_1$-definable from the (finite) collection of all $Y_\sigma$.

We will define the elements $Y_\sigma$ according to two cases. Suppose first that $\sigma$ is such that $Z$ is $n$-thin for $\sigma$, i.e. $H^s_{n+1}(Z_\sigma) \leq 2^{-ns}$.
Define $Y_\sigma = Z_\sigma$ for all such $\sigma$. Note that $\nu_{\sigma} \subseteq Z_{\sigma}$ so our desired condition is satisfied. 

For all other $\sigma \in 2^{<\N}$ of length $n$, it must be the case that $H^s_{n+1}(Z_\sigma) > 2^{-ns}$, which implies that $Z_{\sigma}$ is an element of $\mathcal{S}_{n+1}^{2^{-ns}}.$ By Lemma \ref{H^s_n-inf}, we know that the function $H^s_{n+1}$ is dense above $2^{-ns}$ over $\mathcal{S}_{n+1}^{2^{-ns}}$. As $\mathcal{S}_{n+1}^{2^{-ns}} \cap N_{\nu_{\sigma}}$ is nonempty (it at least contains $Z_\sigma$), we can actually conclude that $H^s_{n+1}$ is dense above $2^{-ns}$ over the closed set $\mathcal{S}_{n+1}^{2^{-ns}} \cap N_{\nu_{\sigma}}$ as well.
Since we have access to $\ACA_0$, we may apply the Dense Monotone Minimum principle (see Theorem \ref{thm:DMMin}) to find an element $Y_{\sigma} \in \mathcal{S}_{n+1}^{2^{-ns}} \cap N_{\nu_\sigma}$ which satisfies $H^s_{n+1}(Y_\sigma) = 2^{-ns}$.

As $Z$ codes an infinite tree, by the definition of the $Z_\sigma$ and the pigeonhole principle, for at least one length-$n$ string $\sigma$, $Z_\sigma$ codes an infinite tree, too, and by the construction above this applies to $Y_\sigma$ as well. Thus, the overall sequence Y is a valid
element of $[S_{T_F}]$. Moreover, by construction the sequence $Y$ extends $\nu$ and is $n$-thin. 

\medskip
\textit{Step 2:} Show that $Y$ is an element of $\mathcal{S}_{n_0}^c$.

\smallskip
By the monotonicity of the function $\tilde{H}^s_{n_0}$, it suffices to show that $\tilde{H}^s_{n_0}(Y^{\leq m}) \geq c$ for all $m \geq n \geq n_0$. Let such an $m$ be arbitrary, and suppose $U$ is a cover witnessing $\tilde{H}^s_{n_0}(Y^{\leq m})$. We wish to show that $W_s(U) \geq c$. 

We may assume that $U$ is prefix-free. We wish to replace $U$ with a cover $U'$ which is valid for computing $\tilde{H}^s_{n_0}(Z^{\leq m})$ and for which $W_s(U') \leq W_s(U).$ Since $Z$ was an element of $\mathcal{S}_{n_0}^c$, this would yield $W_s(U) \geq W_s(U') \geq c$. This would imply that $Y \in \mathcal{S}_{n_0}^c$ since the cover $U$ was arbitrary.

Suppose we have a string $\sigma \in 2^{<\N}$ of length $n$ for which $H^s_{n+1}(Z_{\sigma}) > 2^{-ns}$. If $\sigma$ has no proper extensions in $U$, then we do nothing. Otherwise, delete all of these extensions and replace them with $\sigma$.

We claim that this action will not increase the $s$-weight of the cover. Indeed, let $V_{\sigma}$ denote the (nonempty) set of proper extensions of $\sigma$ in $U$. Note that all strings in $V_{\sigma}$ have length between $n+1$ and $m$, inclusive. Also, if $\tau$ is such that $|\tau| = m$ and $Y_{\sigma}(\tau) = 1$, then $\tau$ must extend some element of $V_{\sigma}$. Given that the original set $U$ was a cover for $Y$, we know that $\tau$ extends some element $\mu \in U$, which must be compatible with $\sigma$. If $\mu \subseteq \sigma$, then $\mu$ would be a proper prefix of any element in $V_{\sigma}$, contradicting the fact that $V_{\sigma}$ was nonempty and $U$ was prefix-free. Therefore, we must have $\mu \supsetneq \sigma$, and thus $\mu \in V_{\sigma}.$

Taken together, we see that $V_{\sigma}$ is a valid cover for computing $\tilde{H}^s_{n+1}(Y_{\sigma}^{\leq m})$. Since $\sigma$ was such that $H^s_{n+1}(Z_{\sigma}) > 2^{-ns}$, the construction of $Y$ yields that $H^s_{n+1}(Y_\sigma) = 2^{-ns}$, and thus $W_s(V_\sigma) \geq 2^{-ns}$. Therefore, replacing the strings in $V_{\sigma}$ with a string $\sigma$ having length $n$ cannot increase the $s$-weight of $U.$

Let $U'$ be the cover obtained from applying this replacement process to all length-$n$ strings $\sigma$ for which $\tilde{H}^s_{n+1}(Z_\sigma) > 2^{-ns}$. It remains to verify that $U'$ is a valid cover for computing $\tilde{H}^s_{n_0}(Z^{\leq m})$. For the length condition, note that we only added strings of length $n$, and we assumed that $n_0 \leq n \leq m$. Thus our string lengths remain within the proper bounds.

For the second condition, suppose we have a string $\tau$ with $|\tau| = m$ and $Z(\tau) = 1$. Let $\sigma$ be the unique length-$n$ prefix of $\tau$. If $\sigma$ satisfies $H^s_{n+1}(Z_\sigma) \leq 2^{-ns}$, then from our construction we know that $Y_{\sigma} = Z_{\sigma}$, and thus $Y_{\sigma}(\tau) = Z_{\sigma}(\tau) = 1$. As the original cover $U$ was valid for computing $\tilde{H}^s_{n_0}(Y^{\leq m})$, this implies that $U$ contains some prefix of $\tau$, which will be compatible with $\sigma$. As $\sigma$ was not one of the strings considered in our replacement process, this compatible prefix $\tau$ will still be in $U'$.

Suppose instead that $\sigma$ satisfies $H^s_{n+1}(Z_\sigma) > 2^{-ns}$. Once again, we know that the original cover $U$ was valid for computing $H^s_{n_0}(Y^{\leq m})$, and thus for computing $H^s_{n_0}(Y_\sigma^{\leq m})$ as well. Given that $\tilde{H}^s_{n+1}(Y_\sigma^{\leq m}) \geq 2^{-ns}$ (from the construction of $Y$) there must be at least one $\tau'$ of length $m$ for which $Y_{\sigma}(\tau') = 1$, so this $\tau'$ has a prefix $\mu \in U$, compatible with $\sigma$. If $\mu \subseteq \sigma$, then $\mu$ is also a prefix of $\tau$ and was not removed in passing to $U'$. If instead $\mu \supseteq \sigma$, then $\mu$ was removed and replaced with $\sigma$.  Either way, we determine that $U'$ contains a prefix of $\tau$, which completes the argument that $U'$ is a valid cover for computing $\tilde{H}^s_{n_0}(Z^{\leq m})$, as desired. 
\end{proof}

We are now ready to complete the proof of our desired density result.

\begin{proposition} \label{BD-ACA}  $\mathsf{BD}$ is provable in $\ACA_0.$
\end{proposition}

\begin{proof}
 By Lemma \ref{BD-fixed} and the remark that followed, we have that $\mathcal{U}_n^d$ is dense in $\mathcal{S}^c_{n_0}$ for all $n \leq n_0$. Assume inductively that $\mathcal{U}_n^d$ is dense in $\mathcal{S}^c_{n_0}$ for some $n \geq n_0.$ Let $\tau \in 2^{<\N}$ be such that $N_\tau \cap \mathcal{S}_{n_0}^c \neq \emptyset$; we wish to exhibit an element $Y \in N_\tau \cap \mathcal{S}^c_{n_0} \cap \mathcal{U}^d_{n+1}$. 

First, we may apply the inductive hypothesis to choose $Z \in N_{\tau} \cap \mathcal{S}_{n_0}^c \cap \mathcal{U}_n^d$. Since $\mathcal{U}_n^d$ is open, we may choose $\tau' \in 2^{<\N}$ such that $\tau \subseteq \tau' \subseteq Z$ and $N_{\tau'} \subseteq \mathcal{U}_n^d$. Then since $N_{\tau'} \cap \mathcal{S}_{n_0}^d$ is nonempty (it at least contains $Z$), by Lemma \ref{T-dense} we can find $Y \in N_{\tau'} \cap \mathcal{S}_{n_0}^d$ with $Y \in \mathcal{T}_n$. By choice of $\tau'$, we have $Y \in \mathcal{T}_n\cap \mathcal{U}_n^d$, so by Lemma \ref{T-U_n} we have $Y \in \mathcal{U}^d_{n+1}$. This proves the density of $\mathcal{U}^d_{n+1}$ in $S_{n_0}^d$, completing the inductive step. We can then use induction on arithmetic formulas in $\ACA_0$ to conclude that $\mathcal{U}_n^d$ is dense in $\mathcal{S}^c_{n_0}$ for all $n \in \N$, as desired. 
\end{proof}

While we have already outlined the argument, we will now formally demonstrate the proof of Besicovitch's Theorem (according to the formulation in Theorem \ref{Bes}) in $\ACA_0$.

\begin{theorem} \label{Bes-ACA} Besicovitch's Theorem $(\BES)$ is provable in $\ACA_0$.
\end{theorem}

\begin{proof}
    First, if $\mathcal{H}^s(F) = \lim_{n \to \infty}H^s_n(Z_F) = c$, then the result is trivial, since we can just take our witnessing subset $E$ to be $F$ itself (with appropriate codes). So assume that $\mathcal{H}^s(F)$ (which may be infinite) is strictly greater than the target measure $c$. In that case, there must be some $n_0 \in \N$ for which $H^s_{n_0}(Z_F) \geq c$. In that case, the closed set $\mathcal{S}^c_{n_0}$ will be nonempty, since it will at least contain the element $Z_F$.

    Now, let $(d_n)_{n \in \N}$ be a strictly decreasing sequence of real numbers which converges to $c$. (For example, we may take $d_n = c + 2^{-n}$.) As each $d_n$ is strictly greater than $c$, it follows from Proposition \ref{BD-ACA} that each open set $\mathcal{U}^{d_n}_{n}$ is dense in $\mathcal{S}^c_{n_0}$, and that this fact is provable in $\ACA_0$. We can then apply $\BCTC$ (also provable in $\ACA_0$) to obtain an element $Z \in \mathcal{S}^c_{n_0} \cap \bigcap_{n \in \N}\mathcal{U}^{d_n}_{n}$. We claim that for the associated closed set $E_Z \subseteq F$, we have
    $$\mathcal{H}^s(E_Z) = \lim_{n \to \infty}H^s_n(Z) = c,$$
    implying that $E_Z$ is our desired witnessing subset.

    First, since $Z \in \mathcal{S}^c_{n_0}$, we know that $H^s_{n_0}(Z) \geq c$, and thus $\mathcal{H}^s(E_Z) \geq c$ since the quantity $H^s_n(Z)$ is monotonically increasing in $n$.  To prove the opposite inequality, let $\epsilon > 0$ be arbitrary, and choose $N \in \N$ so large that $d_n < c + \epsilon $ for all $n \geq N$. For all such $n$, we know that $Z \in \mathcal{U}^{d_n}_n$, and thus 
        $$H^s_{n}(Z) < d_n < c + \epsilon.$$
    As this bound applies to all $n \geq N$, we can conclude that
    $$\mathcal{H}^s(E_Z) = \lim_{n \to \infty}H^s_{n}(Z) \leq c + \epsilon.$$
    But then since $\epsilon > 0$ was arbitrary, we must in fact have $\mathcal{H}^s(E_Z) \leq c$, and therefore $\mathcal{H}^s(E_Z) = c$ when combined with our earlier inequality. This completes the proof.
\end{proof}

Having shown that $\ACA_0$ implies Besicovitch's theorem, it is natural to wonder whether the converse is true. After all, the $\BCTC$ result was integral to this particular proof, and $\BCTC$ itself is equivalent to $\ACA_0$ as shown in Section~\ref{ssec:bct}. While we have not yet determined whether this is the case, the results in the previous section at least demonstrate that Besicovitch - as formalized in Proposition \ref{Bes} - does not hold in $\RCA_0$. We can obtain similar lower bounds by changing the representations of the closed sets involved.

\begin{proposition}
    Over $\RCA_0$, Besicovitch's Theorem implies $\WWKL_0$, as does the analog where the sets $F$ and $E$ are pruned closed. If $F$ is standard closed but $E$ is required to be pruned closed, then the theorem implies $\WKL_0$.
\end{proposition}

\begin{proof}
Assume that Besicovitch's Theorem holds (or one of its variants), and suppose we apply it in the case where $s = 1$. In this case, $H^1_n$ will compute the same function for every $n \in \N$, so for any closed set $F$ with tree code $Z_F$, we will just have $\mathcal{H}^1(F) = H^1_0(Z_F)$. We can then apply the same reversal proofs as were used for $\mathsf{Reg}$-S, $\mathsf{Reg}$-P, and $\mathsf{Reg}$-SP in Proposition \ref{reg-rev}.
\end{proof}

We conclude this section with an examination of Besicovitch's Theorem from the perspective of computability theory. 

Suppose $Z_F \in 2^{\N}$ is any tree code representing a nontrivial standard closed set, and suppose that our real number parameter $s$ and our target measure $c$ are both computable. As Besicovitch's theorem holds in the $\ACA_0$-model
\[
\{X \in 2^{\omega}: X \leq_T Z_F^{(n)} \text{ for some } n \in \omega \},
\]
we can conclude that there exists a witnessing subset $E \subseteq F$ whose tree code $Z_E$ is computable from $Z_F^{(n)}$ for some $n \in \N$. What is the best bound on $n$? Since Besicovitch's Theorem is not provable in $\RCA_0$, we know $n$ is not $0$. By examining the computability theoretic techniques used in previous arguments, we can deduce that a single jump is always sufficient.

\begin{theorem}
Let $Z_F \in 2^{\omega}$ be a tree code representing a nontrivial standard closed set $F$, and suppose that $c, s \geq 0$  are computable real numbers such that 

    $$c \leq \lim_{n \to \infty}H^s_{2^{-n}}(Z_F).$$

    Then there exists a tree code $Z_E \in 2^{\omega}$ for a nontrivial standard closed subset $E \subseteq F$ such that $Z_E \leq_T Z_F'$ and

    $$\lim_{n \to \infty}H^s_n(Z_E) = c.$$    
\end{theorem}

\begin{proof}
    Let the closed set $\mathcal{S}^c_{n_0}$ be defined as in the proof of Theorem \ref{Bes-ACA}. Observe that the corresponding tree
    $$T^c_{n_0} = \{ \nu \in S_{T_F}: \tilde{H}^s_{n_0}(\nu) \geq c\}$$
    is computable from $Z_F$ (i.e. $T_F$). Therefore, the first jump $Z_F'$ is sufficient to compute the pruned subtree
    $$\tilde{T}^c_{n_0} = \{ \sigma \in T^c_{n_0}: \forall m \geq |\sigma|, \exists \tau \supseteq \sigma \text{ with } |\tau| = m \text{ and } \tau \in T^c_{n_0} \},$$
    since $\tilde{T}^c_{n_0}$ is $\Pi^0_1$-definable from $Z_F$. Note that $\tilde{T}^c_{n_0}$ still represents the same closed set $\mathcal{S}^c_{n_0}$, which we can now think of as pruned closed with this new tree code. 

    Now, consider the sequence of open sets $\{\mathcal{U}_n^{d_n}\}_{n \in \omega}$ from the proof of Theorem \ref{Bes-ACA}. Each set $\mathcal{U}_n^{d_n}$ has an open set code $V_n$ defined by $2^{\omega} \setminus T^{d_n}_n$, so these are all (uniformly) computable from $Z_F$. We know from Proposition \ref{BD-ACA} that each open set $\mathcal{U}_n^{d_n}$ is dense in $\mathcal{S}^c_{n_0}$. By following the proof of $\BCTC$-III as given in Proposition \ref{BCTC-II/III}, we can, in a computable way, construct an element $Z \in \mathcal{S}^c_{n_0} \cap \bigcap_{n \in \omega}\mathcal{U}_n^{d_n}$ using the pruned tree $\tilde{T}^c_n$ and the open codes $\{V_n\}_{n \in \omega}$. This element $Z$ is thus computable from $Z_F'$, and by the argument given in the proof of Theorem \ref{Bes-ACA}, the corresponding closed set will satisfy 
    \[
    \lim_{n \to \infty}H^s_n(Z_E) = c,
    \]
as desired.
\end{proof}

Of course, this leaves the question whether Besicovitch's Theorem is provable in a fragment weaker than $\ACA_0$, in particular $\WKL_0$. In this case, we would have an all low model and hence the full power of the Turing jump would not be needed to find a witnessing subset of finite measure.

\printbibliography

@article{simpson2009mass,
	author = {Simpson, Stephen G},
	journal = {Bulletin of Symbolic Logic},
	number = {4},
	pages = {385--409},
	publisher = {Cambridge University Press},
	title = {Mass problems and measure-theoretic regularity},
	volume = {15},
	year = {2009}}

@article{kjos2014finding,
	author = "Kjos-Hanssen, Bj{\o}rn and Reimann, Jan",
	journal = {arXiv preprint arXiv:1408.1999},
	title = {Finding subsets of positive measure},
	year = {2014}}

@book{simpson2009subsystems,
	author = {Simpson, Stephen G},
	publisher = {Cambridge University Press},
	title = {Subsystems of second order arithmetic},
	volume = {1},
	year = {2009}}

@article{brown1990notions,
	author = {Brown, Douglas K},
	journal = {Logic and computation (Pittsburgh, PA, 1987)},
	pages = {39--50},
	publisher = {Contemporary Mathematics},
	title = {Notions of closed subsets of a complete separable metric space in weak subsystems of second-order arithmetic},
	volume = {106},
	year = {1990}}

@phdthesis{yu1987measure,
	author = {Yu, Xiaokang},
	school = {The Pennsylvania State University},
	title = {Measure theory in weak subsystems of second-order arithmetic},
	year = {1987}}

@article{yu1990measure,
	author = {Yu, Xiaokang and Simpson, Stephen G},
	journal = {Archive for Mathematical Logic},
	number = {3},
	pages = {171--180},
	publisher = {Springer},
	title = {Measure theory and weak K{\"o}nig's lemma},
	volume = {30},
	year = {1990}}

@article{brown1993baire,
	author = {Brown, Douglas K and Simpson, Stephen G},
	journal = {The Journal of Symbolic Logic},
	number = {2},
	pages = {557--578},
	publisher = {Cambridge University Press},
	title = {The Baire category theorem in weak subsystems of second-order arithmetic},
	volume = {58},
	year = {1993}}

@inproceedings{davies1952accessibility,
	author = {Davies, Roy O},
	booktitle = {Mathematical Proceedings of the Cambridge Philosophical Society},
	number = {2},
	organization = {Cambridge University Press},
	pages = {215--232},
	title = {On accessibility of plane sets and differentiation of functions of two real variables},
	volume = {48},
	year = {1952}}

@inproceedings{besicovitch1952existence,
	author = {Besicovitch, Abram S},
	booktitle = {Indagationes Mathematicae (Proceedings)},
	organization = {Elsevier},
	pages = {339--344},
	title = {On existence of subsets of finite measure of sets of infinite measure},
	volume = {55},
	year = {1952}}

@book{hirschfeldt2015slicing,
	author = {Hirschfeldt, Denis R},
	publisher = {World Scientific},
	title = {Slicing the truth: On the computable and reverse mathematics of combinatorial principles},
	year = {2015}}

@book{falconer2013fractal,
	author = {Falconer, Kenneth},
	publisher = {John Wiley \& Sons},
	title = {Fractal geometry: mathematical foundations and applications},
	year = {2013}}

@article{pauly2017constructive,
	author = {Pauly, Arno and Fouch{\'e}, Willem},
	journal = {Journal of Logic and Analysis},
	title = {How constructive is constructing measures?},
	volume = {9},
	year = {2017}}

@article{reimann2008effectively,
	author = {Reimann, Jan},
	journal = {Annals of {P}ure and {A}pplied {L}ogic},
	pages = {170-182},
	title = {Effectively closed sets of measures and randomness},
	volume = {156},
	year = {2008}}

@article{Dzhafarov:2022a,
	author = {Dzhafarov, Damir D and Mummert, Carl},
	journal = {Problems, reductions, and proofs. IN: Theory Appl. Comput., Cham},
	title = {Reverse mathematics},
	year = {2022}}

@incollection{Mytilinaios:1996a,
	author = {Mytilinaios, Michael E and Slaman, Theodore A},
	booktitle = {Computability, enumerability, unsolvability},
	pages = {205--218},
	publisher = {Cambridge Univ. Press, Cambridge},
	series = {London Math. Soc. Lecture Note Ser.},
	title = {On a question of {B}rown and {S}impson},
	volume = {224},
	year = {1996}}

@phdthesis{Gruner:2026a,
    author = {Gruner, Emma E},
    title = {Studies in the {R}everse {M}athematics of {B}orel sets and measure regularity},
    school = {The Pennsylvania State University},
    year = {2026} 
}

@article{Brattka:2018a,
	author = {Brattka, Vasco and Hendtlass, Matthew and Kreuzer, Alexander P},
	journal = {Notre Dame J. Form. Log.},
	number = {4},
	pages = {605--636},
	title = {On the uniform computational content of the {B}aire category theorem},
	volume = {59},
	year = {2018}}

\end{document}